\documentclass[12pt]{amsart}
\usepackage{a4wide,enumerate,xcolor}
\usepackage{amsmath,graphicx,comment,enumerate}
\usepackage{mathtools} 
\allowdisplaybreaks

\usepackage{enumitem}
\setlist[itemize]{label={$\bullet$}, leftmargin=30pt, itemsep=3pt}

\let\pa\partial
\let\na\nabla
\let\eps\varepsilon
\newcommand{\N}{{\mathbb N}}
\newcommand{\R}{{\mathbb R}}
\newcommand{\diver}{\operatorname{div}}

\newtheorem{theorem}{Theorem}
\newtheorem{lemma}[theorem]{Lemma}

\begin{document}

\title[A multispecies Keller--Segel system with volume filling]{Analysis of a multispecies cross-diffusion \\ Keller--Segel system with volume filling}

\author[N. Geltner]{Noah Geltner}
\address{Institute of Analysis and Scientific Computing, TU Wien, Wiedner Hauptstra\ss e 8--10, 1040 Wien, Austria}
\email{noah.geltner@tuwien.ac.at} 

\author[A. J\"ungel]{Ansgar J\"ungel}
\address{Institute of Analysis and Scientific Computing, TU Wien, Wiedner Hauptstra\ss e 8--10, 1040 Wien, Austria}
\email{juengel@tuwien.ac.at} 

\author[M. Zhang]{Mingyue Zhang}
\address{Institute of Analysis and Scientific Computing, TU Wien, Wiedner Hauptstra\ss e 8--10, 1040 Wien, Austria}
\email{mingyue.zhang@tuwien.ac.at} 

\date{\today}

\thanks{The authors acknowledge partial support from the Austrian Science Fund (FWF), grant 10.55776/PAT2687825, and from the Austrian Federal Ministry for Women, Science and Research and implemented by \"OAD, project MULT09/2025. This work has received funding from the European Research Council (ERC) under the European Union's Horizon 2020 research and innovation programme, ERC Advanced Grant NEUROMORPH, no.~101018153. 
For open-access purposes, the authors have applied a CC BY public copyright license to any author-accepted manuscript version arising from this submission.} 

\begin{abstract}
A chemotaxis-driven multiphase multispecies diffusion system, arising in the formation of vascular-like structures, is analyzed. The model couples porous-medium-type cross-diffusion equations for the volume fractions of the cellular components with multispecies Keller--Segel equations governing the chemoattractant concentrations, posed in a bounded domain with no-flux boundary conditions. The system is derived within a multiphase framework based on mass and force balance laws, together with a characterization of the mixture pressure gradient, which follows from the volume-filling constraint. The existence of a global weak solution, the weak--strong uniqueness property, the exponential decay to the constant steady state, and the vanishing diffusion limit are established. The existence proof extends the boundedness-by-entropy method to solution codomains that are bounded in some directions only, while the other results are based on various entropy estimates and uniform dissipation bounds. 
\end{abstract}

 

\maketitle


\section{Introduction}

The formation of vascular-like networks is a fundamental process in biological tissue development, tumor growth, and angiogenesis, involving the collective migration and interaction of multiple cell populations guided by chemotactic signaling \cite{BPSO24}. We investigate a multispecies cross-diffusion Keller--Segel system with volume-filling constraint arising from a multiphase modeling framework. The model describes the evolution of the volume fractions of interacting cellular components coupled to the dynamics of chemoattractant concentrations. Starting from mass and force balance laws for the different cell phases, we formally derive a cross-diffusion system with fluxes that contain a porous-medium degeneracy for the cellular volume fractions and a chemotactic drift term. The diffusion matrix incorporates exclusion effects induced by the saturation assumption. The resulting equations combine porous-medium-type degeneracies, chemotactic drift, and strongly coupled cross-diffusion mechanisms, leading to substantial analytical challenges. Such effects are particularly relevant in the description of vascular-like structure formation, where cell motion is constrained by crowding and mechanical interactions within the tissue environment. In this work, we formally derive and rigorously analyze the resulting cross-diffusion chemotaxis system.

\subsection{Model equations}

The volume fractions $u_0,\ldots,u_n$ of the cellular components (or phases) and the chemical signal concentrations $c_1,\ldots,c_g$ are governed by the equations
\begin{align}
  \pa_t u_i &= \diver\bigg(J_i-u_i\sum_{k=0}^n J_k\bigg),
  \quad J_i = \alpha_i u_i^{m_i-1}\na u_i + u_i\sum_{k=1}^g
  \chi_{ik}\na c_k, \label{1.u} \\
  \pa_t c_j &= D_j\Delta c_j - \lambda_j c_j
  + \sum_{k=0}^n b_{jk}u_k\quad\mbox{in }\Omega,\ t>0,\
  i=0,\ldots,n,\ j=1,\ldots,g, \label{1.c} 
\end{align}
together with the initial and no-flux boundary conditions
\begin{align}\label{1.ic}
  & u_i(0)=u_i^0,\quad c_j(0)=c_j^0\quad\mbox{in }\Omega, \quad
  i=0,\ldots,n,\ j=1,\ldots,g, \\
  &\bigg(J_i-u_i\sum_{k=0}^n J_k\bigg)\cdot\nu = \na c_j\cdot\nu = 0
  \quad\mbox{on }\pa\Omega,\ t>0, \label{1.bc} 
\end{align}
where $\Omega\subset\R^d$ ($d\ge 1$) is a bounded domain, $\alpha_i\ge 0$ and $D_j>0$ are diffusion coefficients, $m_i\ge 1$ denotes the porous-medium exponent, originating from the interphase pressure definition, $\chi_{ik}\in\R$ describes the sensitivity of the cellular components modeling direct aggregation ($\chi_{ik}<0$) or repulsion ($\chi_{ik}>0$) of the $i$th cell species, $b_{jk}\ge 0$ is the cell production rate, and $\lambda_j>0$ is the degradation rate of the signals. The volume fraction $u_0$ is often associated with the water phase, which may not diffuse ($\alpha_0=0$). The mixture is assumed to be saturated, so the volume-filling constraint $\sum_{i=0}^n u_i=1$ holds. This condition implies the incompressibility of the mixture, since the barycentric velocity $v=\sum_{i=0}^n(J_i-u_i\sum_{k=0}^n J_k)$ vanishes and is, in particular, divergence-free.

Equations \eqref{1.u} are derived formally from mass and force balances by using the multiphase approach of \cite{LKBJS06}. The viscous stress tensor of the $i$th phase consists of the phase-specific pressure, being the sum of the mixture pressure and the intraphase pressure in porous-medium-type form. The forces are given as the sum of interphase, viscous drag, and chemotactic forces. The volume-filling constraint allows us to remove the mixture pressure, which can be interpreted as a Lagrange multiplier. Then the fluxes can be computed explicitly, leading to the form \eqref{1.u} of the mass balance equations. Details are given in Appendix \ref{sec.deriv}.  

The originality of the present paper stems from the following key aspects. System \eqref{1.u}--\eqref{1.bc} combines several analytical features that have so far been studied largely in isolation, namely nonlinear cross diffusion, porous-medium degeneracies, strong coupling between all components, and a mixed volume-filling structure in which only part of the variables is subject to exclusion effects. The interaction of these mechanisms leads to a highly nonlinear diffusion system that falls outside the scope of existing entropy methods. The main analytical observation of this work is the proof that entropy methods remain sufficiently flexible to accommodate this level of complexity. To this end, we develop an extension of the boundedness-by-entropy method of \cite{Jue15} to cross-diffusion systems with partial volume filling, thereby considerably broadening the class of models accessible to this technique. In addition, the governing equations \eqref{1.u}--\eqref{1.c} arise naturally from mass and force balance laws within a multiphase framework, resulting in a thermodynamically consistent model endowed with an intrinsic entropy structure. This physical derivation not only provides a rigorous modeling foundation but also explains the origin of the Lyapunov functional that underpins the analytical framework developed in this paper and detailed in the following subsection.

\subsection{Entropy structure}

Equations \eqref{1.u}--\eqref{1.c} can be formulated as the cross-diffusion system
\begin{align*}
  \pa_t q_i = \diver\bigg(\sum_{j=0}^{n+g} 
  \bar A_{ij}(\bar q)\na q_j\bigg) + \bar r_i(\bar q), \quad
  i=0,\ldots,n+g,
\end{align*}
where $\bar q=(q_0,\ldots,q_{n+g})=(\bar u,c)=(u_0,\ldots,u_n,c_1,\ldots,c_g)$, $\bar A(\bar q)$ is the diffusion matrix, and $\bar r$ is the reaction term originating from the source terms in \eqref{1.c}. The key of our analysis is the observation that this system possesses an entropy structure. This means that there exists a so-called entropy density $h$ such that the matrix $\bar h''(\bar q)\bar A(\bar q)$ is positive definite, where $\bar h''$ is the Hessian of the entropy density. Because of the volume-filling constraint, the matrix $\bar h''(\bar q)\bar A(\bar q)$ in the variable $\bar q$ is positive definite only on a subspace of $\R^{n+1+g}$. This issue is overcome by replacing $u_0$ by $1-\sum_{i=1}^n u_i$ and to work with the variable $q=(u,c)=(u_1,\ldots,u_n,c_1,\ldots,c_g)$, similarly as in \cite{JuSt13}. Formulating $\bar h''(\bar q)\bar A(\bar q)$ in terms of the reduced variable $q$ removes the rank deficiency and yields a positive definite matrix on $\R^{n+g}$. In contrast to \cite{JuSt13}, we work with a variable that satisfies a {\em partial} volume-filling constraint. 

To make the entropy structure more explicit, we introduce 
\begin{align}\label{1.hbar}
    \bar h(\bar q) = \sum_{i=0}^n u_i(\log u_i-1)
  - \sum_{i=0}^n\sum_{j=1}^g\frac{\chi_{ij}}{D_j}u_ic_j
  + \frac{K}{2}\sum_{j=1}^g c_j^2.
\end{align}
This function is convex if $K>0$ is sufficiently large. A formal computation shows, along solutions to \eqref{1.u}--\eqref{1.bc}, that
\begin{align}\label{1.ei}
  \frac{d}{dt}&\int_\Omega\bar h(\bar q)dx
  + \int_\Omega(\na\bar q)^T:\bar h''(\bar q)\bar A(\bar q)
  \na\bar q dx \\
  &= -\int_\Omega\sum_{j=1}^g
  \bigg(\sum_{k=0}^n b_{jk}u_k - \lambda_jc_j\bigg)
  \bigg(\sum_{i=0}^n\frac{\chi_{ij}}{D_j}u_i - Kc_j\bigg)dx,
  \nonumber 
\end{align}
where $A:B$ is the Frobenius product between two matrices $A$ and $B$. 
Thanks to the volume-filling constraint, the volume fractions $u_i$ are bounded, which implies that the right-hand side of \eqref{1.ei} is bounded uniformly in $(u,c)$. The difficulty is to estimate the entropy production term (the second integral on the left-hand side of \eqref{1.ei}), since $\bar h''(\bar q)\bar A(\bar q)$ is positive semidefinite only. By estimating with the reduced variable $q$, we show that for sufficiently large $K$, the entropy production term is estimated, for some $C_0>0$, according to 
\begin{align}\label{1.ep}
  \int_\Omega(\na\bar q)^T:\bar h''(\bar q)\bar A(\bar q)\na\bar q dx
  \ge C_0\int_\Omega\bigg(\sum_{i=0}^n\frac{\alpha_i}{2}
  |\na u_i^{m_i/2}|^2 + \sum_{j=1}^g|\na c_j|^2\bigg)dx,
\end{align}
providing uniform gradient bounds for $u_i^{m_i/2}$ (if $\alpha_i >0$) and for $c_j$. 


\subsection{State of the art}

In the case of two phases and one signal ($n=g=1$), the choice $\alpha_0=\chi_{01}=b_{10}=0$ and $\alpha_1=1$ in equations \eqref{1.u}--\eqref{1.c} leads to a vanishing flux of the zeroth phase, $J_0=0$, and to the volume-filling Keller--Segel system for $u=u_1$, $c=c_1$:
\begin{equation}\label{1.uc}
\begin{aligned}
  \pa_t u &= \diver\big((1-u)u^{m_1-1}\na u
  + \chi_{11}(1-u)u\na c\big), \\
  \pa_t c &= D_1\Delta c - \lambda_1 c + b_{11}u
  \quad\mbox{in }\Omega,\ t>0,
\end{aligned}
\end{equation}
which was investigated in \cite{GJZ26}. The mathematical study of chemotaxis began with the pioneering contributions of Patlak in the 1950s \cite{Pat53} and Keller and Segel in the 1970s \cite{KeSe70}. Over the last decades, the Keller--Segel equations generated a substantial body of mathematical research. Rather than attempting a comprehensive review, we highlight only a few contributions, restricting ourselves to degenerate and volume-filling models.

A distinctive property of degenerate Keller--Segel models of the type
\begin{align*}
  \pa_t u = \diver\big(u^{m-1}\na u - u\na c\big), \quad
  \pa_t c = \Delta c - c + bu\quad\mbox{in }\Omega,\ t>0,
\end{align*}
is the finite speed of propagation. The existence of solutions for a similar system (not including degenerate diffusion at $u=0$) with the optimal range $m>2-2/d$ was proven in \cite{TaWi12}. The convergence to possibly nonconstant steady states for $m > 2$, and to the constant steady state in the case $m=2$, was established in \cite[Theorem 1.2]{Jia18} for sufficiently small $b>0$. Volume-filling Keller--Segel equations were derived from a lattice model in \cite[Sec.~3]{HiPa02}, but this class of models is different from \eqref{1.uc}. Combined degenerate and volume-filling models have been studied too; see, e.g., \cite{BKU07,WWW12}. Pattern formation for the model
\begin{align*}
  \pa_t u = \diver\big(u^{m-1}\na u - u(1-u)\na c\big), \quad
  \pa_t c = \Delta c - c + u\quad\mbox{in }\Omega,\ t>0,
\end{align*}
was established in \cite{PeZh24} for different parameter regimes of $m$ and with small masses. 

Over the last decades, there has been significant progress in the development and analysis of multispecies chemotaxis models. The parabolic--elliptic model for multiple bacteria populations was analyzed in \cite{VPA23}. Extensions to parabolic--parabolic systems and both multiple cell densities and chemical concentrations were suggested in \cite{Lin24,Wol02}. The tendency of a population $i$ towards a population $k$ is quantified in \cite{Wol02} by the parameter $\eta_{ik}:=\sum_{j=1}^g\chi_{ij}b_{jk}$. A priori estimates from a Lyapunov functional are obtained in the conflict-free case, i.e., $\eta_{ik}\eta_{ki}>0$ for all $i,k$. Based on these estimates, the existence of a global weak solution to the conflict-free system under a sub-critical mass condition was proved in \cite{Lin24}. In the work \cite{KRZ18}, population species indexed by a continuous parameter and a single chemical signal was suggested. The qualitative behavior (blow-up, nontrivial stationary solutions, asymptotic behavior) was investigated for two-signal models in \cite{TaWa13} and for two-density models in \cite{NeTe14}, for instance. 

Our results do not rely on an explicit Lyapunov functional like in \cite{Lin24,Wol02} but on the entropy functional $\bar h$, for which uniform bounds can be derived. The volume-filling constraint, which guarantees bounded solutions $u_i$, allows us, in contrast to \cite{Lin24}, to dispense with any restrictions on the model coefficients. For our analysis, we extend the boundedness-by-entropy method of \cite{Jue15} to solution codomains that are bounded in some directions only.


\subsection{Main results}

We impose the following hypotheses:
\begin{itemize}
\item[(H1)] Domain: $\Omega\subset\R^d$ ($d\ge 1$) is a bounded domain, $T>0$; set $\Omega_T:=\Omega\times(0,T)$.
\item[(H2)] Initial data: $u^0\in L^\infty(\Omega;\R^n)$, $c^0\in H^1(\Omega;\R^g)$ satisfy $0\le c^0\le 1$ in $\Omega$ and $u^0(x)\in\overline{\mathcal D}$ for a.e.\ $x\in\Omega$, where $\mathcal D=\{u\in(0,1)^{n}:\sum_{i=1}^n u_i<1\}$. 
\item[(H3)] Parameters: $m_i\ge 1$; $\alpha_0\ge 0$, $\alpha_i>0$ for $i=1,\ldots,n$; $D_j$, $\lambda_j>0$; $b_{ji}\ge 0$, $\chi_{ij}\in\R$ for $i=0,\ldots,n$, $j=1,\ldots,g$. 
\end{itemize}

First, we show the existence of a global weak solution to \eqref{1.u}--\eqref{1.bc}. 

\begin{theorem}[Existence of solutions]\label{thm.ex}
Let Hypotheses (H1)--(H3) hold and let $\kappa\in\mathbb{N}$, $\kappa>d/2$. Then there exists a weak solution $q=(u,c)$ to \eqref{1.u}--\eqref{1.bc} with $u_0=1-\sum_{i=1}^n u_i$, satisfying $u(x,t)\in\overline{\mathcal D}$, $c(x,t)\in[0,\infty)^g$  for a.e.\ $(x,t)\in\Omega_T$,
\begin{align*}
  & u_i\in L^\infty(\Omega_T), \quad
  \alpha_iu_i^{m_i/2}, c_j\in L^2(0,T;H^1(\Omega)), \\
  &\pa_t u_i,\, \pa_t c_j\in L^2(0,T;H^\kappa(\Omega)')\quad\mbox{for }
  i=0,\ldots,n,\ j=1,\ldots,g.
\end{align*}
\end{theorem}

The proof of this theorem relies on expressions \eqref{1.ei} and \eqref{1.ep}, leading to the entropy inequality
\begin{align*}
  \frac{d}{dt}\int_\Omega\bar h(\bar q)dx
  + C_0\int_\Omega\bigg(\sum_{i=0}^n\frac{\alpha_i}{2}
  |\na u_i^{m_i/2}|^2 + \sum_{j=1}^g|\na c_j|^2\bigg)dx
  \le C_1 + C_1\int_\Omega\bar h(\bar q)dx
\end{align*}
for some constants $C_0$, $C_1>0$. We infer that the entropy $\int_\Omega \bar h(\bar q)dx$ is bounded on finite time intervals, and the entropy dissipation gives gradient bounds for $u_i^{m_i/2}$ and $c_j$. The boundedness-by-entropy method of \cite{Jue15} cannot be applied directly, since the solution codomain $\mathcal D\times [0,\infty)^g$ is not bounded, but the technique of the proof (fixed-point and compactness arguments) can still be used. A technical difficulty is the proof of the entropy production inequality \eqref{1.ep}; see Section \ref{sec.ent}. 

Our second result is the weak--strong uniqueness of solutions. The aim is to show that, whenever a strong solution exists, it is unique within the class of weak solutions sharing the same initial data. Currently, a proof of the uniqueness of weak solutions is out of reach for general cross-diffusion systems because of the strong gradient coupling and low regularity. For the weak--strong uniqueness result, we consider a simpler entropy density than \eqref{1.hbar} to streamline the computations:
\begin{align}\label{1.h2}
  h_2(\bar q) = \sum_{i=0}^n u_i(\log u_i-1) 
  + \frac{K}{2}\sum_{j=1}^g c_j^2, \quad
  \bar q=(\bar u,c)\in(0,1)^{n+1}\times\R^g,
\end{align}
where $K>0$ is some number. The result reads as follows.

\begin{theorem}[Weak--strong uniqueness]\label{thm.wsu}
Let $1\le m_i\le 2$ for $i=0,\ldots,n$. Let $q=(u,c)$ be a weak solution to \eqref{1.u}--\eqref{1.bc} satisfying the entropy inequality \eqref{4.ei}, and let $\widehat q=(\widehat u,\widehat c)$ be another weak solution to \eqref{1.u}--\eqref{1.bc} with the same initial data, satisfying the entropy equality
\begin{align}\label{1.ee}
  \int_\Omega h_2(\widehat{\bar q}(t))dx
  &+ \int_0^t\int_\Omega(\na\widehat{\bar q})^T:h_2''(\widehat{\bar q})
  \bar A(\widehat{\bar q})\na\widehat{\bar q} dxds \\
  &= \int_\Omega h_2(\bar q(0))dx
  + K\int_0^t\int_\Omega\sum_{j=1}^g \widehat c_j\bigg(\sum_{k=0}^n
  b_{jk}\widehat u_k - \lambda_j \widehat c_j\bigg)dxds, \nonumber 
\end{align}
fulfilling the additional regularity $\na\log\widehat u_i\in L^\infty(0,T;L^\infty(\Omega))$ and the strict positivity $\widehat u_i\ge\mu>0$ in $\Omega_T$ for some $\mu>0$, for $i=0,\ldots,n$ (called ``strong'' solution). We also assume that $\pa\Omega\in C^{2}$ and $\widehat c^0 \in W^{2,q}(\Omega)$ for $q>d$. Then $q_i(t)=\widehat q_i(t)$ in $\Omega$ for $t>0$, $i=1,\dots,n+g$.
\end{theorem}

The proof of Theorem \ref{thm.wsu} is based on an estimation of the relative entropy density, comparing the weak solution $\bar{q}$ and the ``strong'' solution $\widehat{\bar q}$:
\begin{align}\label{1.rel}
  h(\bar q|\widehat{\bar q}) = \sum_{i=0}^n\bigg(u_i\log\frac{u_i}{\widehat u_i} 
  - u_i + \widehat u_i\bigg) 
  + \frac{K}{2}\sum_{j=1}^g(c_j-\widehat c_j)^2,
\end{align}
where $\bar q=(\bar u,c)$, $\widehat{\bar q}=(\widehat{\bar u},\widehat c)$ and $\bar u=(u_0,\ldots,u_n)$, $\widehat{\bar u}=(\widehat u_0,\ldots,\widehat u_n)$.
The aim is the proof of
\begin{align}\label{1.relei}
  \frac{d}{dt}\int_\Omega h(\bar q|\widehat{\bar q})dx
  \le C(K)\int_\Omega|\bar{u}-\widehat{\bar u}|^2dx
  \le 2C(K)\int_\Omega h(\bar q|\widehat{\bar q})dx.
\end{align}
We conclude from Gronwall's lemma that $\int_\Omega h(\bar q|\widehat{\bar q}) dx=0$. The strict positivity of $\widehat u_i$ implies that $h(\bar q|\widehat{\bar q})\ge \frac12|u-\widehat{\bar u}|^2$, and hence $\bar u(t)=\widehat{\bar u}(t)$ in $\Omega$ for $t>0$. A first difficulty is the proof of \eqref{1.relei}. Since the mobility matrix $B(\bar q)=\bar A(\bar q)h_2''(\bar q)^{-1}$ is only positive definite on the subspace $Z=\{\bar z\in\R^{n+1}:\sum_{k=0}^n z_k=0\}\times\R^g$, we need to work with a certain projection operator $\Pi_1$, making the proof quite technical. 

A second difficulty is the degeneracy. The weak--strong uniqueness argument only works when the degeneracy is not too ``large'' \cite{HeJu26}, which leads to the restriction $m_i\le 2$. The reason is that we need to absorb the expression $\int_\Omega u_i^{2}|\Pi_1(\na\log(u/\widehat u))_i|^2 dx$, coming from some mixed terms, by the entropy production $\int_\Omega u_i^{m_i}|\Pi_1(\na\log(u/\widehat u))_i|^2 dx$, which is only possible if $m_i\le 2$ (since then $u_i^2\le u_i^{m_i}$). 

Next, we prove that the weak solution $(u,c)$ converges for $t\to\infty$ to the constant steady state
\begin{align*}
  u_i^\infty = \frac{1}{|\Omega|}\int_\Omega u_i^0 dx, \quad
  c_j^\infty = \frac{1}{\lambda_j}\sum_{k=0}^n b_{jk}u_k^\infty \quad
  \mbox{for }i=0,\ldots,n,\ j=1,\ldots,g,
\end{align*}
if the diffusion coefficients $\alpha_i$ or $D_j$ are sufficiently large or the coefficients $b_{ij}$ or $|\chi_{ij}|$ are sufficiently small in the sense specified in the following theorem. Let $b'=(b_{ji}-b_{j0})_{ij}$, $\chi'=(\chi_{ij}-\chi_{0j})_{ij}$ and set
\begin{align}\label{1.paradef}
  \alpha_* = \min_{i=1,\ldots,n}\alpha_i, \quad
  D_* = \min_{j=1,\ldots,g}D_j, \quad 
  \lambda_* = \min_{j=1,\ldots,g}\lambda_j.
\end{align}
Furthermore, we set $\bar q^\infty=(\bar u^\infty,c^\infty)$. We denote by $C_P>0$ the constant of the Poincar\'e inequality and by $\|\cdot\|_2$ the spectral norm.

\begin{theorem}[Long-time behavior]
Assume that $1\le m_i\le 2$ and 
\begin{align*}
  \frac{C_P^2\|b'\|_2^2\|\chi'\|_2^2}{4\alpha_*^2 D_*\lambda_*} < 1.
\end{align*}
Let $(u,c)$ be a weak solution to \eqref{1.u}--\eqref{1.bc} satisfying the entropy inequality \eqref{4.ei}. Then there exist constants $C>0$ and $\mu>0$ depending on the initial data and the model parameters such that
\begin{align*}
  \sum_{i=1}^n\|u_i(t)-u_i^\infty\|_{L^2(\Omega)}^2
  + \sum_{j=1}^g\|c_j(t)-c_j^\infty\|_{L^2(\Omega)}^2
  \le Ce^{-\mu t}, \quad t>0.
\end{align*}
\end{theorem}

The proof is again based on the relative entropy, replacing the second argument by the steady state. We prove in Section \ref{sec.time} that for suitable $K>0$ and $\eps>0$,
\begin{align*}
  \frac{d}{dt}\int_\Omega h(\bar q(t)|\bar q^\infty)dx 
  &\le -C_P^{-2}(\alpha_*-\eps)
  \int_\Omega\sum_{i=1}^n(u_i-u_i^\infty)^2 dx \\
  &\phantom{xx}- K\bigg(\lambda_*-\frac{C_P^2\|b'\|_2^2\|\chi'\|_2^2}{4\eps^2
  D_*} \bigg)\int_\Omega\sum_{j=1}^g(c_j-c_j^\infty)^2 dx.
\end{align*}
Choosing $\sqrt{C_P^2\|b'\|_2^2\|\chi'\|_2^2/(4 D_*\lambda_*)}<\eps<\alpha_*$, Gronwall's lemma concludes the result.

Our final result concerns the vanishing diffusion limit. We show that the solution $(u_\eps,c_\eps)$ to \eqref{1.u}--\eqref{1.bc} from Theorem \ref{thm.ex} with $\eps=D_j$ converges, up to a subsequence, as $\eps\to 0$ to a solution to the limiting system
\begin{equation}\label{1.lim}
\begin{aligned}
  \pa_t u_i &= \diver\bigg(J_i-u_i\sum_{k=0}^n J_k\bigg),
  \quad J_i = \alpha_i u_i^{m_i-1}\na u_i + u_i\sum_{k=1}^g
  \chi_{ik}\na c_k, \\
  \pa_t c_j &= -\lambda_jc_j + \sum_{k=0}^n b_{jk}u_k
  \quad\mbox{in }\Omega,\ t>0,\ i=0,\ldots,n,\ j=1,\ldots,g,
\end{aligned}
\end{equation}
together with the initial condition (3) and the no-flux boundary condition for the cell density equations in \eqref{1.bc}.

\begin{theorem}[Vanishing diffusion limit]\label{thm.lim}
Let $\pa\Omega\in C^{1,1}$, $1\le m_i\le 2$ for $i=0,\ldots,n$, and let $\bar q^\eps=(\bar u^\eps,c^\eps)$ be a weak solution to \eqref{1.u}--\eqref{1.bc} with $D_j=\eps>0$ for $j=1,\ldots,g$. Then there exists a subsequence that is not relabeled such that, as $\eps\to 0$, for $i=0,\ldots,n$, $j=1,\ldots,g$,
\begin{align*}
  u_i^\eps\to u_i,\quad c_j^\eps\to c_j, \quad
  \eps\Delta c_j^\eps\to 0 
  \quad\mbox{strongly in }L^2(\Omega_T),
\end{align*}
where $(u,c)$ solves \eqref{1.lim} with initial conditions \eqref{1.ic} and no-flux boundary conditions for the cell densities in \eqref{1.bc}. 
\end{theorem}

For the proof, we use the entropy density
\begin{align}\label{1.h3}
  h_3(\bar q) 
  = \sum_{i=0}^n u_i(\log u_i-1)
  + \frac{K}{2}\sum_{j=1}^g|\na c_j|^2 + n + 1, \quad
  \bar q = (\bar u,c),
\end{align}
and show that, for some constant $C>0$ independent of $\eps$,
\begin{align*}
  \frac{d}{dt}\int_\Omega h_3(\bar q^\eps)dx 
  + \int_\Omega\bigg(\sum_{i=0}^n
  \frac{\alpha_i}{2}|\na u_i^\eps|^2
  + \eps\sum_{j=1}^g(\Delta c_j^\eps)^2\bigg)dx
  \le C\int_\Omega h_3(\bar q^\eps)dx,
\end{align*}
which proves, after an application of Gronwall's lemma, some estimates for $(u^\eps,c^\eps)$ uniform in $\eps$. The convergence results then follow from the Aubin--Lions lemma.

The paper is organized as follows. The entropy structure, in particular the positive definiteness of the mobility matrix, is explored in Section \ref{sec.ent}. Theorems \ref{thm.ex}--\ref{thm.lim} are proved in Sections \ref{sec.ex}--\ref{sec.lim}. Finally, equations \eqref{1.u} are derived formally from a multiphase approach in Appendix \ref{sec.deriv}, giving a justification of our model. 


\section{Entropy structure}\label{sec.ent}

We formulate equations \eqref{1.u}--\eqref{1.c} as a system of equations for the vector-valued unknown $\bar q=(\bar u,c)$, where $\bar u=(u_0,\ldots,u_n)$ and $c=(c_1,\ldots,c_g)$:
\begin{align*}
  & \pa_t\bar q_i = \diver\bigg(\sum_{j=0}^{n+g}\bar A_{ij}(\bar q)
  \na\bar q_j\bigg) + \bar r_i(\bar q), \quad i=0,\ldots,n+g, \\
  & \mbox{where }\bar A(\bar q) = \begin{pmatrix}
  \bar Q^{11}(\bar q) & \bar Q^{12}(\bar q) \\
  \bar Q^{21}(\bar q) & \bar Q^{22}(\bar q)
  \end{pmatrix}, \nonumber 
\end{align*}
the submatrices $Q^{k\ell}(\bar q)$ are given by
\begin{equation}\label{3.barQ}
\begin{aligned}
  \bar Q^{11}_{ij}(\bar q) &= \alpha_iu_i^{m_i-1}\delta_{ij}
  -u_i\alpha_ju_j^{m_j-1} &&\mbox{for }i=0,\ldots,n,\ j=0,\ldots,n, \\
  \bar Q^{12}_{ij}(\bar q) &= u_i\chi_{ij}-u_i\sum_{k=0}^n u_k\chi_{kj}
  &&\mbox{for }i=0,\ldots,n,\ j=1,\ldots,g, \\
  \bar Q^{21}_{ij}(\bar q) &= 0 &&\mbox{for }i=1,\ldots,g,\  
  j=0,\ldots,n, \\
  \bar Q^{22}_{ij}(\bar q) &= D_i\delta_{ij} 
  &&\mbox{for }i=1,\ldots,g,\ j=1,\ldots,g,
\end{aligned}
\end{equation}
and the reaction term $\bar r(\bar q)=(r_i(\bar q))_{i=0}^{n+g}$ reads as
\begin{align}\label{3.r}
  r_i(\bar q) = 0\quad\mbox{for }i=0,\ldots,n, \quad
  r_{j+n}(\bar q) = \sum_{k=0}^n b_{jk}u_k - \lambda_jc_j
  \quad\mbox{for }j=1,\ldots,g.
\end{align}
We set $r(\bar q)=(r_{n+1},\ldots,r_{n+g})(\bar q)$, which gives $\bar r = (0,r)$.

Taking into account the volume-filling constraint, we can replace $u_0$ by $1-\sum_{i=1}^n u_i$ and introduce the reduced entropy density
\begin{align*}
  h(q) = \bar h(u_0,q), \quad\mbox{where }q=(u,c),\ 
  u=(u_1,\ldots,u_n),\ u_0 = 1-\sum_{i=1}^n u_i,
\end{align*}
defined for $q\in\mathcal D\times\R^g$, where $\bar h$ is introduced in \eqref{1.hbar} and $\mathcal D$ is defined in Hypothesis (H2). 

\begin{lemma}\label{lem.h}
The mapping $h':\mathcal D\times\R^g\to\R^{n+g}$ is a bijection for sufficiently large $K>0$.
\end{lemma}

\begin{proof}
We first compute the partial derivatives of $h$:
\begin{align*}
  \frac{\pa h}{\pa u_i}(q) &= \log\frac{u_i}{u_0} - R_i(c)
  \quad\mbox{for }i=1,\ldots,n, \\
  \frac{\pa h}{\pa c_j}(q) &= 
  - \sum_{i=0}^n\frac{\chi_{ij}}{D_j}u_i + Kc_j 
  \quad\mbox{for }j=1,\ldots,g,
\end{align*}
where $R_i(c)=\sum_{j=1}^g(\chi_{ij}-\chi_{0j})c_j/D_j$. For given $w\in\R^{n+g}$, we claim that there exists $q=(u,c) \in \mathcal{D}\times\R^g$ such that $h'(q)=w$. This equation is equivalent to
\begin{align*}
  u_i = \frac{\exp(w_i+R_i(c))}{1 + \sum_{k=1}^n\exp(w_k+R_k(c))}, \quad
  c_j = \frac{w_{j+n}}{K} + \frac{1}{K}\sum_{i=1}^n
  \frac{\chi_{ij}-\chi_{0j}}{D_j}u_i,
\end{align*}
where $i=1,\ldots,n$, $j=1,\ldots,g$. We wish to find $(u,c)$ such that both equations hold simultaneously. We accomplish this by using a fixed-point argument. For this, we define $F:\R^g\to\R^g$ by
\begin{align*}
  F_j(c) = \frac{w_{j+n}}{K} + \frac{1}{K}\sum_{i=1}^n
  \frac{\chi_{ij}-\chi_{0j}}{D_j}
  \frac{\exp(w_i+R_i(c))}{1 + \sum_{k=1}^n\exp(w_k+R_k(c))},
  \quad j=1,\ldots,g.
\end{align*}
The derivative $F'$ can be estimated as $|F'(c)|\le C(\chi,D)/K$. This shows that $F$ is Lipschitz continuous:
\begin{align*}
  |F(c) - F(\widehat{c})|\le \frac{C(\chi,D)}{K}
  |c-\widehat{c}|\quad\mbox{for }c,\,\widehat{c}\in\R^g.
\end{align*}
Thus, choosing $K$ sufficiently large, the mapping $F$ is a contraction on $\R^g$ and admits a unique fixed point by Banach's fixed-point theorem. 
\end{proof}

Next, we show that the matrix $\bar h''(\bar q)\bar A(\bar q)$ is positive definite on a certain subset.

\begin{lemma}[Positive definiteness]\label{lem.hA1}
Let $K>0$ be sufficiently large. For $\bar q\in(0,1)^{n+1}\times\R^g$, the matrix $\bar h''(\bar q)\bar A(\bar q)$ is positive definite on the subspace $Z:=\{\bar z=(z_k)\in\R^{n+1+g}:\sum_{k=0}^{n}z_k=0\}$, i.e., there exists a constant $\gamma(K)>0$ such that for all $\bar q\in(0,1)^{n+1}\times\R^g$ and $\bar z\in Z$,
\begin{align*}
  \bar z^T\bar h''(\bar q)\bar A(\bar q)\bar z
  \ge \sum_{i=0}^n\frac{\alpha_i}{2} u_i^{m_i-2}z_i^2
  + \gamma(K)\sum_{j=1}^{g}z_{j+n}^2.
\end{align*}
It holds that $\gamma(K)\to \infty$ as $K\to\infty$. 
\end{lemma}

\begin{proof}
We introduce the matrix $Y\in\R^{(n+1)\times g}$ by $Y_{ij}=-\chi_{ij}/D_j$ for $i=0,\ldots,n$, $j=1,\ldots,g$. Then
\begin{align*}
  \bar h''(\bar q) = \begin{pmatrix}
  \mbox{diag}(1/u_0,\ldots,1/u_n) & Y \\
  Y^T & \mbox{diag}(K,\ldots,K)
  \end{pmatrix}.
\end{align*}
Let $\bar z=(\xi,\eta)=(\xi_0,\ldots,\xi_n,\eta_1,\ldots,\eta_g)\in Z$. The constraint $\sum_{i=0}^n \xi_i=0$ implies that
\begin{align}\label{3.11}
  \xi^T\mbox{diag}\bigg(\frac{1}{u_0},\ldots,\frac{1}{u_n}\bigg)
  \bar Q^{11}\xi &= \sum_{i=0}^n\alpha_i u_i^{m_i-2}\xi_i^2, \\
  \xi^T\mbox{diag}\bigg(\frac{1}{u_0},\ldots,\frac{1}{u_n}\bigg)
  \bar Q^{12}\eta &= \sum_{i=0}^n\sum_{j=1}^g\chi_{ij}\xi_i\eta_j
  = -\xi^T Y\bar Q^{22}\eta. \label{3.12}
\end{align}
Taking into account $\bar Q^{21}=0$, this shows that
\begin{align}\label{3.aux}
  \bar z^T\bar h''(\bar q)\bar A(\bar q)\bar z
  &= \xi^T\mbox{diag}\bigg(\frac{1}{u_0},\ldots,\frac{1}{u_n}\bigg)
  \bar Q^{11}\xi 
  + \xi^T\bigg(\mbox{diag}\bigg(
  \frac{1}{u_0},\ldots,\frac{1}{u_n}\bigg)\bar Q^{12} 
  + Y\bar Q^{22}\bigg)\eta \\
  &\phantom{xx}+ \eta^TY^T\bar Q^{11}\xi 
  + \eta^T\big(Y^T\bar Q^{12} 
  + \mbox{diag}(K,\ldots,K)\bar Q^{22}\big)\eta.
  \nonumber  
\end{align}
The first term on the right-hand side simplifies thanks to \eqref{3.11}; the second term vanishes because of \eqref{3.12}; and we use for the remaining terms
\begin{align*}
  |\eta^TY^T\bar Q^{12}\eta| \le C_0\sum_{j=1}^g \eta_j^2,
  \quad\eta^T\mbox{diag}(K,\ldots,K)\bar Q^{22}\eta
  = K\sum_{j=1}^g D_j\eta_j^2,
\end{align*}
where $C_0>0$ depends only on $\chi_{ij}$ and $D_j$. We apply Young's inequality with $\delta>0$ to the third term in \eqref{3.aux}:
\begin{align*}
  |\eta^TY^T\bar Q^{11}\xi| &= \bigg|\sum_{i=0}^n\sum_{j=1}^g
  \eta_j\frac{\chi_{ij}}{D_j}\bigg(\alpha_iu_i^{m_i-1}\xi_i  - u_i\sum_{k=0}^n\alpha_ku_k^{m_k-1}\xi_k\bigg)\bigg| \\
  &\le \delta\sum_{i=0}^n \alpha_iu_i^{2(m_i-1)}\xi_i^2
  + C(\delta)\sum_{j=1}^g\eta_j^2
  \le \delta \sum_{i=0}^n \alpha_iu_i^{m_i-2}\xi_i^2
  + C(\delta)\sum_{j=1}^g\eta_j^2,
\end{align*}
where $C(\delta)>0$ also depends on $\chi_{ij}$, $D_j$, and $\alpha_i$. 
We conclude from \eqref{3.aux} that
\begin{align*}
  \bar z^T\bar h''(\bar q)\bar A(\bar q)\bar z
  \ge (1-\delta)\sum_{i=0}^n \alpha_i u_i^{m_i-2}\xi_i^2
  + \sum_{j=1}^g(D_jK - C(\delta))\eta_j^2.
\end{align*}
We choose $\delta=1/2$ and $\gamma(K)=D_*K-C(\delta)$ (see \eqref{1.paradef}) to finish the proof.
\end{proof}

While the full variable $\bar q$ simplifies the computation of the matrix $\bar h''(\bar q)\bar A(\bar q)$, the existence analysis is based on a system in which the unknown $u_0$ is replaced by $1-\sum_{i=1}^n u_i$, thanks to the volume-filling constraint. Accordingly, let $q=(u,c)\in\R^{n+g}$ with $u=(u_1,\ldots,u_n)$ and $c=(c_1,\ldots,c_g)$. We introduce the reduced diffusion matrix $A(q)\in\R^{(n+g)\times(n+g)}$ by 
\begin{align*}
  A_{ij}(q) = \begin{cases}
  \bar A_{ij}(u_0,q)-\bar A_{i0}(u_0,q) &\mbox{for }i,j=1,\ldots,n, \\
  \bar A_{ij}(u_0,q) &\mbox{else}
\end{cases}
\end{align*} 
for $i,j=1,\ldots,n+g$. We claim that $h''(q)A(q)$ is positive definite.

\begin{lemma}[Positive definiteness of the reduced matrix]\label{lem.hA2}
Let $K>0$ be sufficiently large. For $q\in\mathcal D\times\R^g$, the matrix $h''(q)A(q)$ is positive definite, i.e., with the constant $\gamma(K)>0$ from Lemma \ref{lem.hA1}, for all $q\in\mathcal D\times\R^g$ and $z\in\R^{n+g}$,
\begin{align*}
  z^T h''(q)A(q)z
  \ge \sum_{i=0}^n\frac{\alpha_i}{2} u_i^{m_i-2}z_i^2
  + \gamma(K)\sum_{j=1}^{g}z_{j+n}^2,
\end{align*}
where $u_0=1-\sum_{i=1}^nu_i$ and $z_0=-\sum_{i=1}^nz_i$ in the first sum. 
\end{lemma}

\begin{proof}
We define the matrix $G\in\R^{(n+1+g)\times(n+g)}$ by $G_{ij}=-1$ for $i=0$, $j=1,\ldots,n$ and $G_{ij}=\delta_{ij}$ else. This matrix is simply the Jacobian of the mapping $q\mapsto(1-\sum_{i=1}^nu_i,u_1,\ldots,u_n,c)$. Therefore, by the chain rule, $h''(q) = G^T\bar h''(\bar q)G$. Next, let $z\in\R^{n+g}$ and set $z_0=-\sum_{i=1}^nz_i$. Then $\bar z=(z_0,\ldots,z_{n+g})\in Z$ (see Lemma \ref{lem.hA1}) and $Gz=\bar z$. It follows from the definition of $A(q)$ that, for $i=1,\ldots,n+g$,
\begin{align*}
  (A(q)z)_i &= \sum_{k=1}^n(\bar A_{ik}-\bar A_{i0})z_k
  + \sum_{k=n+1}^{n+g}\bar A_{ik}z_k
  = \sum_{k=1}^{n+g}\bar A_{ik}z_k - \bar A_{i0}\sum_{k=1}^n z_k \\
  &= \sum_{k=0}^{n+g}\bar A_{ik}z_k = (\bar A(\bar q)\bar z)_i,
\end{align*}
and consequently, $(GA(q)z)_i = (A(q)z)_i = (\bar A(\bar q)\bar z)_i$ for $i=1,\ldots,n+g$. If $i=0$, we deduce from the definition of $\bar A(\bar q)$ that $\bar A_{0k} = - \sum_{j=1}^n\bar A_{jk}$ for $k=0,\ldots,n+g$
\begin{align*}
  (GAz)_0 = -\sum_{j=1}^n(A(q)z)_j 
  = -\sum_{j=1}^n(\bar A(\bar q)\bar z)_j
  = -\sum_{j=1}^n\sum_{k=0}^{n+g}\bar A_{jk}z_k
  = \sum_{k=0}^{n+g}\bar A_{0k}z_k = (\bar A(\bar q)\bar z)_0.
\end{align*}
Therefore, we have
\begin{align*}
  z^Th''(q)A(q)z = z^T(G^T\bar h''(\bar q)G)A(q)z
  = (Gz)^T\bar h''(\bar q)(GA(q))z
  = \bar z^T\bar h''(\bar q)\bar A(\bar q)\bar z,
\end{align*}
and the result follows from Lemma \ref{lem.hA1}. 
\end{proof}


\section{Proof of the existence theorem}\label{sec.ex}

In this section, we prove Theorem \ref{thm.ex}. To this end, we introduce the entropy variables
\begin{align*}
  w_i &= \frac{\pa h}{\pa u_i} = \log\frac{u_i}{u_0}
  - \sum_{k=1}^g\frac{\chi_{ik}-\chi_{0k}}{D_k}c_k
  \quad\mbox{for } i=1,\ldots,n, \\
  w_{j+n} &= \frac{\pa h}{\pa c_j} 
  = -\sum_{k=0}^n\frac{\chi_{kj}}{D_j}u_k + Kc_j
  \quad\mbox{for }j=1,\ldots,g,
\end{align*}
where $u_0=1-\sum_{i=1}^n u_i$. Then, with the mobility matrix $B(q)=A(q)h''(q)^{-1}$, we can formulate \eqref{1.u}--\eqref{1.c} as
\begin{align*}
  \pa_t q_i = \diver\bigg(\sum_{j=1}^{n+g}B_{ij}(q)\na w_j\bigg)
  + r_i(q)\quad\mbox{in }\Omega,\ t>0,\ i=1,\ldots,n+g,
\end{align*}
recalling definition \eqref{3.r} of $r_i(q)=r_i(\bar q)$. The mobility matrix is positive semidefinite, since by Lemma \ref{lem.hA2}, for $z\in\R^{n+g}$,
\begin{align*}
  z^T B(q)z &= \big(h''(q)^{-1}z\big)^T h''(q)A(q)
  \big(h''(q)^{-1}z\big) \\
  &\ge \sum_{i=0}^n \frac{\alpha_i}{2}u_i^{m_i-2}\big(h''(q)^{-1}z\big)_i^2
  + \gamma(K)\sum_{j=1}^g\big(h''(q)^{-1}z\big)_{j+n}^2.
\end{align*}

We construct now a time-discrete and regularized problem. We approximate the initial datum such that $u^0\in\mathcal{D}$. The general case $u^0\in\overline{\mathcal{D}}$ can be obtained by introducing first $u_\delta^0 = (u^0+\delta)/(1+\delta(n+1))\in\mathcal{D}$ and later passing to the limit $\delta\to 0$. Let $T>0$, $N\in\N$, and $\tau=T/N$. We choose $\kappa\in\N$ such that $\kappa>d/2$, ensuring the compactness of the embedding $H^\kappa(\Omega)\hookrightarrow L^\infty(\Omega)$. Let $w^{k-1}\in L^\infty(\Omega;\R^{n+g})$. If $k=1$, we set $w^0=h'(u^0,c^0)$. We want to find a solution $w^k\in H^\kappa(\Omega;\R^{n+g})$ to the following variational problem:
\begin{align}\label{3.approx}
  \frac{1}{\tau}\int_\Omega&\big((h')^{-1}(w^k)-(h')^{-1}(w^{k-1})\big)
  \cdot\phi dx + \int_\Omega\na\phi^T:B(w^k)\na w^k dx \\
  &+ \eps b(w^k,\phi) = \int_\Omega r((h')^{-1}(w^k))\cdot\phi dx
  \nonumber 
\end{align}
for all $\phi\in H^\kappa(\Omega;\R^{n+g})$, where $\eps>0$. The functional $b(w^k,\phi)$ is defined by
\begin{align*}
  b(w^k,\phi) = \int_\Omega\bigg(\sum_{|\beta|=\kappa}
  D^\beta w^k\cdot D^\beta\phi + w^k\cdot\phi\bigg)dx,
\end{align*}
where $D^\beta$ is a partial derivative of order $\kappa$ and $\beta\in\N^d$ is a multiindex. 

We want to apply \cite[Lemma 5]{Jue15}. The reaction term satisfies
\begin{align*}
  r(q)\cdot h'(q) = -\sum_{j=1}^g\bigg(\sum_{k=0}^n b_{jk}u_k
  - \lambda_jc_j\bigg)
  \bigg(\sum_{i=0}^n\frac{\chi_{ij}}{D_j}u_i - Kc_j\bigg) \le C_r
\end{align*}
for some $C_r>0$ and all $q=(u,c)\in\mathcal D\times\R^g$. This estimate as well as the statements of Lemmas \ref{lem.h} and \ref{lem.hA2} show that the hypotheses of \cite[Lemma 5]{Jue15} are satisfied. Consequently, we obtain the existence of a weak solution $w^k\in H^\kappa(\Omega;\R^{n+g})$ to \eqref{3.approx} satisfying $u(w^k(x))\in\mathcal{D}$ for $(u(w^k(x)),c(w^k(x))):=(h')^{-1}(w^k(x))$, and the discrete entropy inequality
\begin{align}\label{3.dei}
  \int_\Omega h(u^k,c^k)dx
  &+ \tau\int_\Omega(\na w^k)^T:B(w^k)\na w^k dx 
  + \eps b(w^k,w^k) \\
  &\le C_r\tau\mbox{meas}(\Omega)
  + \int_\Omega h(u^{k-1},c^{k-1})dx \nonumber 
\end{align}
holds, where $u^k=u(w^k)$ and $c^k=c(w^k)$. Inequality \eqref{3.dei} yields estimates uniform in $\eps$ and $\tau$. Indeed, define the piecewise constant functions
\begin{align*}
  w^{(\tau)}(x,t) = w^k(x), \quad
  q^{(\tau)}(x,t) = (u^{(\tau)},c^{(\tau)})(x,t) 
  = q^k(x) = (u^k(x),c^k(x))
\end{align*}
for $x\in\Omega$, $t\in((k-1)\tau,k\tau]$, $k=1,\ldots,N$. At time $t=0$, we set $w^{(\tau)}(0)=w^0=h'(u^0,c^0)$ and $q^{(\tau)}(0)=q^0=(u^0,c^0)$. We introduce the shift operator
\begin{align*}
  (\sigma_\tau q^{(\tau)})(x,t) = q^{k-1}(x) = (h')^{-1}(w^{k-1}(x)).
\end{align*}

\begin{lemma}\label{lem.est1}
There exists $C>0$ independent of $(\eps,\tau)$ such that 
\begin{align*}
  \|u^{(\tau)}\|_{L^\infty(\Omega_T)}
  + \|c^{(\tau)}\|_{L^2(0,T;H^1(\Omega))}
  + \sqrt\eps\|w^{(\tau)}\|_{L^2(0,T;H^\kappa(\Omega))} &\le C, \\
  \sum_{i=0}^n\alpha_i\|(u_i^{(\tau)})^{m_i/2}\|_{L^2(0,T;H^1(\Omega))}
  &\le C.
\end{align*}
\end{lemma}

\begin{proof}
The lemma follows from the discrete entropy inequality \eqref{3.dei} after summation over $k=1,\ldots,\ell$ with $\ell\tau\le T$:
\begin{align*}
  \int_\Omega h(q^k)dx
  &+ \tau\sum_{k=1}^\ell\int_\Omega(\na w^k)^T:B(w^k)\na w^k dx
  + \eps\tau\sum_{k=1}^\ell b(w^k,w^k) \\
  &\le C_r\tau\ell\mbox{meas}(\Omega) + \int_\Omega h(q^0)dx. \nonumber 
\end{align*}
By the generalized Poincar\'e inequality,
\begin{align*}
  \eps b(w^k,w^k) \ge \eps C\|w^k\|_{H^\kappa(\Omega)}^2. 
\end{align*}
We deduce from Lemma \ref{lem.hA2} and 
\begin{align*}
  \int_\Omega(\na w^k)^T:B(w^k)\na w^k dx
  &= \int_\Omega(\na q^k)^T:h''(q^k)A(q^k)\na q^k dx \\
  &\ge \int_\Omega\bigg(\sum_{i=0}^n\frac{\alpha_i}{2}
  (u_i^k)^{m_i-2}|\na u_i^k|^2 + \gamma(K)
  \sum_{j=1}^g|\na c_j^k|^2\bigg)dx 
\end{align*}
the bounds for $(u_i^{(\tau)})^{m_i/2}$ and $c_j^{(\tau)}$ in $L^2(0,T;H^1(\Omega))$. 
\end{proof}

\begin{lemma}\label{lem.est2}
There exists $C>0$ independent of $(\eps,\tau)$ such that 
\begin{align*}
  \|u^{(\tau)}-\sigma_\tau u^{(\tau)}\|_{L^2(0,T;H^\kappa(\Omega)')}
  + \|c^{(\tau)}-\sigma_\tau c^{(\tau)}\|_{L^2(0,T;H^\kappa(\Omega)')}
  \le \tau C.
\end{align*}
\end{lemma}

\begin{proof}
Let $\phi\in L^2(0,T;H^\kappa(\Omega;\R^{n+g}))$. It follows from the weak formulation \eqref{3.approx} that
\begin{align*}
  \frac{1}{\tau}&\bigg|\int_0^T\int_\Omega
  (q^{(\tau)} - \sigma_\tau q^{(\tau)})\cdot\phi dxdt\bigg|
  \le \big\|A(q^{(\tau)})\na q^{(\tau)}\big\|_{L^2(\Omega_T)}
  \|\na\phi\|_{L^2(\Omega_T)} \\
  &\phantom{xx}+ \eps\|w^{(\tau)}\|_{L^2(0,T;H^\kappa(\Omega))}
  \|\phi\|_{L^2(0,T;H^\kappa(\Omega))}
  + \|r(q^{(\tau)})\|_{L^2(\Omega_T)}\|\phi\|_{L^2(\Omega_T)}.
\end{align*}
The last two terms are uniformly bounded. For the remaining term, we observe that
\begin{align*}
  (A(q^{(\tau)})\na q^{(\tau)})_i 
  &= \alpha_i(u_i^{(\tau)})^{m_i-1}\na u_i^{(\tau)}
  + u_i^{(\tau)}\sum_{j=1}^g\chi_{ij}\na c_j^{(\tau)} 
  - u_i^{(\tau)}\sum_{k=1}^n \alpha_k(u_k^{(\tau)})^{m_k-1}
  \na u_k^{(\tau)} \\
  &\phantom{xx}+ u_i^{(\tau)}\sum_{k=1}^n\alpha_0
  (u_0^{(\tau)})^{m_0-1}\na u_k^{(\tau)}
  - u_i^{(\tau)}\sum_{k=0}^n u_k^{(\tau)}\sum_{j=1}^g
  \chi_{kj}\na c_j^{(\tau)}. \nonumber 
\end{align*} 
The bounds of Lemma \ref{lem.est1} imply that $(u_k^{(\tau)})^{m_k-1} \na u_k^{(\tau)} = (2/m_k)(u_k^{(\tau)})^{m_k/2} \na(u_k^{(\tau)})^{m_k/2}$ and $\na c_j^{(\tau)}$ are uniformly bounded in $L^2(\Omega_T)$. Therefore, all terms except the fourth one on the right-hand side are uniformly bounded. We reformulate the fourth term:
\begin{align*}
  u_i^{(\tau)}\sum_{k=1}^n\alpha_0
  (u_0^{(\tau)})^{m_0-1}\na u_k^{(\tau)}
  = -\alpha_0 u_i^{(\tau)}(u_0^{(\tau)})^{m_0-1}\na u_0^{(\tau)}
  = -\frac{2\alpha_0}{m_0}u_i^{(\tau)}
  (u_0^{(\tau)})^{m_0/2}\na(u_0^{(\tau)})^{m_0/2},
\end{align*}
and this expression is uniformly bounded in $L^2(\Omega_T)$ as well. We have proved that the family $A(q^{(\tau)})\na q^{(\tau)}$ is uniformly bounded in $L^2(\Omega_T)$. Therefore, for some constant $C>0$ independent of $(\eps,\tau)$,
\begin{align*}
  \frac{1}{\tau}\bigg|\int_0^T\int_\Omega
  (q^{(\tau)} - \sigma_\tau q^{(\tau)})\cdot\phi dxdt\bigg|
  \le C\|\na\phi\|_{L^2(\Omega_T)},
\end{align*}
and taking the supremum over all $\phi$ such that $\|\phi\|_{L^2(0,T;H^\kappa(\Omega))}=1$ finishes the proof.
\end{proof}

The uniform estimates of Lemmas \ref{lem.est1} and \ref{lem.est2} allow us to perform the limit $(\eps,\tau)\to 0$. In particular, since $m_i\ge 1$, we can apply the Aubin--Lions lemma in the version of \cite[Theorem 3]{CJL14} to obtain the existence of a subsequence that is not relabeled such that
\begin{align*}
  u_i^{(\tau)}\to u_i, \quad c_j^{(\tau)}\to c_j\quad  
  \mbox{strongly in }L^2(\Omega_T)\mbox{ as }(\eps,\tau)\to 0
\end{align*}
for $i=0,\ldots,n$, $j=1,\ldots,g$. Thanks to the $L^\infty(\Omega_T)$ bound, the strong convergence for $u_i^{(\tau)}$ also holds in $L^r(\Omega_T)$ for all $r<\infty$. Moreover, it holds that
\begin{align*}
  \alpha_i \na(u_i^{(\tau)})^{m_i/2}\rightharpoonup \alpha_i\na u_i^{m_i/2}
  &\quad\mbox{weakly in }L^2(\Omega_T),\ i=0,\ldots,n, \\
  \na c_j^{(\tau)}\rightharpoonup\na c_j
  &\quad\mbox{weakly in }L^2(\Omega_T),\ j=1,\ldots,g, \\
  \tau^{-1}(u_i^{(\tau)}-\sigma_\tau u_i^{(\tau)})
  \rightharpoonup\pa_t u_i 
  &\quad\mbox{weakly in }L^2(0,T;H^\kappa(\Omega)'),\ i=0,\ldots,n, \\
  \tau^{-1}(c_j^{(\tau)}-\sigma_\tau c_j^{(\tau)})
  \rightharpoonup\pa_t c_j 
  &\quad\mbox{weakly in }L^2(0,T;H^\kappa(\Omega)'),\ j=1,\ldots,g, \\
  \eps w_i^{(\tau)}\to 0 
  &\quad\mbox{strongly in }L^2(0,T;H^\kappa(\Omega)).
\end{align*}
These convergences are sufficient to pass to the limit in the diffusion term:
\begin{align*}
  A(q^{(\tau)})\na q^{(\tau)}\rightharpoonup A(q)\na q
  \quad\mbox{weakly in }L^1(\Omega_T),
\end{align*}
and in view of the uniform boundedness of $A(q^{(\tau)})\na q^{(\tau)}$ in $L^2(\Omega_T)$, this weak convergence also holds in $L^2(\Omega_T)$. Hence, we can pass to the limit $(\eps,\tau)\to 0$ in \eqref{3.approx} to conclude that $q=(u,c)$ is a solution to
\begin{align*}
  \int_0^T\langle \pa_t q,\phi\rangle dt
  + \int_0^T\int_\Omega\na\phi^T:A(q)\na q dxdt
  = \int_0^T\int_\Omega r(q)\cdot\phi dxdt
\end{align*}
for all $\phi\in L^2(0,T;H^\kappa(\Omega;\R^{n+g}))$ and, by density, also for all $\phi\in L^2(0,T;H^1(\Omega;\R^{n+g}))$. The initial datum is satisfied in the sense of $H^\kappa(\Omega)'$. The restriction to approximating initial data with values in $\mathcal{D}$ is removed by passing to the de-regularization limit, yielding the result for all $u^0 \in \overline{\mathcal{D}}$.

By definition of the entropy density $h$, we have $u^{(\tau)}(x,t) = (\pa h/\pa u)^{-1}(w^{(\tau)}(x,t))\in\mathcal D$ and thus, in the limit, $u(x,t)\in\overline{\mathcal{D}}$ for a.e.\ $(x,t)\in\Omega_T$. We deduce from $b_{jk}u_k\ge 0$ that $\pa_t c_j - D_j\Delta c_j + \lambda_j c_j\ge 0$ and we conclude from the minimum principle and $c_j^0\ge 0$ that $c_j\ge 0$ in $\Omega_T$. This finishes the proof of Theorem \ref{thm.ex}. 


\section{Proof of the weak--strong uniqueness property}

First, we verify that the solution constructed in Theorem \ref{thm.ex} satisfies an entropy inequality with the entropy density $h_2$, defined in \eqref{1.h2}. 

\begin{lemma}
Let $1\le m_i\le 2$ and let $\bar q=(\bar u,c)$ be the weak solution to \eqref{1.u}--\eqref{1.bc} constructed in Theorem \ref{thm.ex} with $\bar u=(u_0,\ldots,u_n)$ and $u_0=1-\sum_{i=1}^n u_i$. Then, for $t\in(0,T)$,
\begin{align}\label{4.ei}
  \int_\Omega h_2(\bar q(t))dx
  &+ \int_0^t\int_\Omega(\na\bar q)^T:h_2''(\bar q)
  \bar A(\bar q)\na\bar q dxds \\
  &\le \int_\Omega h_2(\bar q(0))dx
  + K\int_0^t\int_\Omega\sum_{j=1}^g c_j\bigg(\sum_{k=0}^n
  b_{jk}u_k - \lambda_j c_j\bigg)dxds. \nonumber 
\end{align}
\end{lemma}

\begin{proof}
We formulate \eqref{3.approx} in the variable $\bar q^{(\tau)}=(\bar u^{(\tau)},c^{(\tau)})$, i.e., with $h_2$ instead of $h$ and use the test function $\bar w^{(\tau)}=h'_2(\bar q^{(\tau)})$, where 
\begin{align*}
  \bar w_i=\log u_i^{(\tau)}\mbox{ for }i=0,\ldots,n, \quad
  \bar w_{j+n} =Kc_{j}^{(\tau)}\mbox{ for }j=1,\ldots,g.
\end{align*}
Then, setting $\bar q^{(\tau)}=(\bar u^{(\tau)},c^{(\tau)})=(h_2')^{-1}(\bar w^{(\tau)})$, 
\begin{align*}
  \frac{1}{\tau}\int_\Omega&\big(\bar q^{(\tau)}(t) 
  - \bar q^{(\tau)}(t-\tau)\big)
  \cdot \bar h'(\bar q^{(\tau)}(t))dx
  + \int_\Omega (\na\bar w^{(\tau)})^T:
  \bar B(\bar q^{(\tau)})\na\bar w^{(\tau)}dx \\
  &\le  K\int_\Omega\sum_{j=1}^g c_j^{(\tau)}
  \bigg(\sum_{i=0}^n b_{ji}u_i^{(\tau)} 
  - \lambda_j c_j^{(\tau)}\bigg)dx,
\end{align*}
where $\bar B(\bar q^{(\tau)}) = \bar A(\bar q^{(\tau)})h_2''(\bar u^{(\tau)})^{-1}$. We deduce from $\sum_{i=0}^n\na u_i^{(\tau)}=0$ that
\begin{align*}
  F(\bar q^{(\tau)}) 
  &= (\na\bar w^{(\tau)})^T:\bar B(\bar q^{(\tau)})
  \na\bar w^{(\tau)}
  = (\na q^{(\tau)})^T:\bar h_2''(q^{(\tau)})
  \bar A(\bar q^{(\tau)})\na\bar q^{(\tau)} \\
  &= \sum_{i=0}^n \frac{4\alpha_i}{m_i^2}|\na(u_i^{(\tau)})^{m_i/2}|^2
  + \sum_{i=0}^n\sum_{j=1}^g\frac{2\chi_{ij}}{m_i}
  (u_i^{(\tau)})^{1-m_i/2}\na(u_i^{(\tau)})^{m_i/2}
  \cdot\na c_j^{(\tau)} \\
  &\phantom{xx}+ K\sum_{j=1}^g D_j|\na c_j^{(\tau)}|^2.
\end{align*}
For sufficiently large $K>0$, the functional $F$ is a positive definite quadratic form with respect to $\na(u_i^{(\tau)})^{m_i/2}$ and $\na c_j^{(\tau)}$. Observe that $(u_i^{(\tau)})^{1-m_i/2}$ is well defined uniformly in $\tau$ if $1\le m_i\le 2$. Tonelli's theorem in the version of \cite[Sec.~8.2]{Eva10} implies that $F$ is weakly lower semicontinuous, showing that 
\begin{align*}
  \liminf_{\tau\to 0}\int_\Omega F(\bar q^{(\tau)}) dx
 \ge \int_\Omega(\na\bar q)^T:\bar h_2''(q)\bar A(\bar q)\na\bar q dx.
\end{align*}
It follows from the convexity of $h_2$ that
\begin{align*}
  \frac{1}{\tau}\int_\Omega&\big(h_2(\bar q^{(\tau)}(t))
  - h_2(\bar q^{(\tau)}(t-\tau))\big)dx
  + \int_\Omega F(\bar q^{(\tau)}) dx \\
  &\le K\int_\Omega\sum_{j=1}^g c_j^{(\tau)}
  \bigg(\sum_{k=0}^n b_{jk}u_k^{(\tau)} 
  - \lambda_j c_j^{(\tau)}\bigg)dx.
\end{align*}
Therefore, integrating over $t$ and taking the limit inferior, we obtain the desired entropy inequality. 
\end{proof}

We proceed to the proof of Theorem \ref{thm.wsu}. Let $\bar q=(\bar u,c)$ be a weak solution and $\widehat{\bar q}=(\widehat{\bar u},\widehat c)$ be a ``strong'' solution to \eqref{1.u}--\eqref{1.bc}. Taking into account entropy inequality \eqref{4.ei} for $\bar q$, entropy equality \eqref{1.ee} for $\widehat{\bar q}$, as well as
\begin{align*}
  \int_\Omega&(\bar q-\widehat{\bar q})
  \cdot\bar h_2'(\widehat{\bar q})dx\Big|_t
  - \int_\Omega(\bar q-\widehat{\bar q})
  \cdot h_2'(\widehat{\bar q})dx\Big|_0 \\
  &= \int_0^t\big\langle\pa_t h_2'(\widehat{\bar q}),
  \bar q-\widehat{\bar q}\big\rangle ds
  + \int_0^t\big\langle\pa_t(\bar q-\widehat{\bar q}),
  h_2'(\widehat{\bar q})\big\rangle ds,
\end{align*}
we obtain for the relative entropy \eqref{1.rel}:
\begin{align*}
  \int_\Omega& h(\bar q|\widehat{\bar q})(t)dx 
  - \int_\Omega h(\bar q|\widehat{\bar q})(0)dx\\
  &\le -\int_0^t\int_\Omega \big((\na h_2'(\bar q))^T
  :\bar B(\bar q)\na\bar h_2'(\bar q)
  - (\na\bar h_2'(\widehat{\bar q}))^T:\bar B(\widehat{\bar q})
  \na h_2'(\widehat{\bar q})\big)dxds \\
  &\phantom{xx}- \int_0^t\langle\pa_t h_2'(\widehat{\bar q}),
  \bar q-\widehat{\bar q}\rangle ds
  - \int_0^t\big\langle\pa_t(\bar q - \widehat{\bar q}),
  h_2'(\widehat{\bar q})\big\rangle ds \\
  &\phantom{xx}+ K\int_0^t\int_\Omega(r(\bar q)\cdot c 
  - r(\widehat{\bar q})\cdot\widehat{c})dxds,
\end{align*}
where $\bar B(\widehat{\bar q}) = \bar A(\widehat{\bar q})( h_2'')^{-1}(\widehat{\bar q})$. After inserting the evolution equations, regrouping the gradient terms, and simplifying the reaction terms, we arrive at
\begin{align}\label{4.I123}
  \int_\Omega& h(\bar q|\widehat{\bar q})(t)dx 
  - \int_\Omega h(\bar q|\widehat{\bar q})(0)dx \\
  &\le -\int_0^t\int_\Omega \big((\na h_2'(\bar q))^T
  :\bar B(\bar q)\na h_2'(\bar q)
  - (\na h_2'(\widehat{\bar q}))^T:\bar B(\widehat{\bar q})
  \na h_2'(\widehat{\bar q})\big)dxds \nonumber \\
  &\phantom{xx}+ \int_0^t\int_\Omega\big[\na\big(h_2''(\widehat{\bar q})
  (\bar q-\widehat {\bar q})\big)\big]^T:
  \bar B(\widehat{\bar q})\na h_2'(\widehat{\bar q})dxds 
  \nonumber \\
  &\phantom{xx}+ \int_0^t\int_\Omega(\na h_2'(\widehat{\bar q}))^T
  :\big(\bar B(\bar q)\na h_2'(\bar q) 
  - \bar B(\widehat{\bar q})\na h_2'(\widehat{\bar q})\big)
  dxds \nonumber \\
  &\phantom{xx}+ K\int_0^t\int_\Omega\big(r(\bar q)\cdot c
  - r(\widehat{\bar q})\cdot\widehat c 
  - r(\widehat{\bar q})\cdot(c-\widehat c)
  - (r(\bar q) - r(\widehat{\bar q}))
  \cdot\widehat c\big)dxds \nonumber \\
  &= I_1+I_2+I_3, \nonumber 
\end{align}
where $r(\bar q)=(r_{n+1},\ldots,r_{n+g})(\bar q)$ and 
\begin{align*}
  I_1 &= -\int_0^t\int_\Omega(\na(h_2'(\bar q)
  - h_2'(\widehat{\bar q})))^T:\bar B(\bar q)
  \na(h_2'(\bar q) - h_2'(\widehat{\bar q}))dxds, \\
  I_2 &= -\int_0^t\int_\Omega\big[(\na(h_2'(\bar q)
  - h_2'(\widehat{\bar q})))^T:\bar B(\bar q)
  - \big(\na(h_2''(\widehat{\bar q})(\bar q-\widehat {\bar q}))\big)^T
  :\bar B(\widehat{\bar q})\big]\na h_2'(\widehat{\bar q})dxds, \\
  I_3 &= K\int_0^t\int_\Omega(r(\bar q) 
  - r(\widehat{\bar q}))\cdot(c-\widehat c)dxds.
\end{align*}

Before estimating the terms $I_1$, $I_2$, and $I_3$, we introduce the subspace 
\begin{align*}
  L := \mbox{span}\{\bar u\}^\perp\times\R^g 
  = \bigg\{y\in\R^{n+1+g}: \sum_{i=0}^n u_iy_i = 0\bigg\}.
\end{align*}
Let $\bar y\in L$ and set $z_i=u_iy_i$ for $i=0,\ldots,n$, $z_j=y_j$ for $j=n+1,\ldots,n+g$. Then $\bar z=(z_0,\ldots,z_{n+g})\in Z$, and Lemma \ref{lem.hA1}, applied to $h_2$, implies that
\begin{align*}
  \bar y^T\bar B(\bar q)\bar y 
  \ge \sum_{i=0}^n\frac{\alpha_i}{2} u_i^{m_i}y_i^2
  + \gamma(K)\sum_{j=1}^g y_{n+j}^2.
\end{align*}
Furthermore, we introduce the projection operator $\Pi:\R^{n+1+g}\to L$ by
\begin{align*}
  \Pi(\bar y) = (\Pi_1(\bar y),\Pi_2(\bar y))
  = \bigg(\bigg(y_i - \sum_{k=0}^n u_ky_k\bigg)_{i=0}^n,
  (y_j)_{j=n+1}^{n+g}\bigg)^T
  = \bar y - \sum_{k=0}^n u_ky_k
  \begin{pmatrix} \mathbf{1} \\ \mathbf{0} \end{pmatrix},
\end{align*}
where $\mathbf{1}\in\R^{n+1}$ and $\mathbf{0}\in\R^g$. Since $\mathbf{1}^T\bar Q^{11}=0$ and $\mathbf{1}^T\bar Q^{12}=0$, we have
\begin{align*}
  \Pi(\bar y)^T\bar A(\bar q) = \bar y^T \bar A(\bar q).
\end{align*}

We continue with the estimation of $I_1$. Set $\bar z = \na( h_2'(\bar q) - h_2'(\widehat{\bar q}))$. Then $z_i=\na\log(u_i/\widehat u_i)$ for $i=0,\ldots,n$. We have
\begin{align}\label{4.zPi}
  \bar z-\Pi(\bar z) = \bigg(\bigg(\sum_{k=0}^n
  u_kz_k\bigg)_{i=0}^n,0,\ldots,0\bigg)\in 
  \mbox{span}\{(\mathbf{1},\mathbf{0})^T\},
\end{align}
and consequently, $(\bar z-\Pi(\bar z))^T\bar B(\bar q)=0$. This shows that
\begin{align*}
  I_1 &= -\int_0^t\int_\Omega\bar z^T:\bar B(\bar q)\bar z dxds 
  = -\int_0^t\int_\Omega \Pi(\bar z)^T:\bar B(\bar q)\bar zdxds \\
  &= -\int_0^t\int_\Omega\Pi(\bar z)^T:\bar B(\bar q)\Pi(\bar z)dxds
  - \int_0^t\int_\Omega\Pi(\bar z)^T:\bar B(\bar q)
  (\bar z - \Pi(\bar z))dxds \\
  &\le -\int_0^t\int_\Omega\bigg(
  \sum_{i=0}^n\frac{\alpha_i}{2} u_i^{m_i}
  |\Pi(\bar z)_i|^2 +\gamma(K)
  \sum_{j=1}^g|\na(c_j-\widehat c_j)|^2\bigg)dxds \\
  &\phantom{xx}- \int_0^t\int_\Omega\sum_{i,j=0}^n \Pi(\bar z)_i
  \big(\alpha_i u_i^{m_i}\delta_{ij} - u_i\alpha_ju_j^{m_j}\big)
  \bigg(\sum_{k=0}^n u_kz_k\bigg) dxds.
\end{align*}
Since $\Pi(\bar z)\in L$, we see from the identity $\sum_{i=0}^n\Pi(\bar z)_iu_i=0$ that the $u_i\alpha_ju_j^{m_j}$ term is eliminated. It follows from $\sum_{k=0}^nu_k\na\log u_k = \sum_{k=0}^n \na u_k=0$ that
\begin{align}\label{4.uz}
  \sum_{k=0}^n u_kz_k = \sum_{k=0}^n u_k\na\log\frac{u_k}{\widehat u_k}
  = -\sum_{k=0}^nu_k\na\log\widehat u_k
  = \sum_{k=0}^n(\widehat u_k-u_k)\na\log\widehat u_k.
\end{align}
Hence,
\begin{align*}
  I_1 &\le -\int_0^t\int_\Omega\bigg(
  \sum_{i=0}^n\frac{\alpha_i}{2} u_i^{m_i}
  |\Pi(\bar z)_i|^2 + \gamma(K)\sum_{j=1}^g
  |\na(c_j-\widehat c_j)|^2\bigg)dxds \\
  &\phantom{xx}- \int_0^t\int_\Omega\sum_{i=0}^n
  \alpha_i\Pi(\bar z)_iu_i^{m_i}
  \bigg(\sum_{k=0}^n(\widehat u_k-u_k)\na\log\widehat u_k\bigg)dxds.
\end{align*}
Observing that $\Pi(\bar z)_i=\na\log(u_i/\widehat u_i)-\sum_{k=0}^n u_k\na\log(u_k/\widehat u_k)$ for $i=0,\ldots,n$ and applying Young's inequality, we obtain
\begin{align*}
  I_1 &\le -\int_0^t\int_\Omega
  \sum_{i=0}^n\frac{3\alpha_i}{8} u_i^{m_i}|\Pi(\bar z)_i|^2 dxds
  - \gamma(K)\int_0^t\int_\Omega\sum_{j=1}^g
  |\na(c_j-\widehat c_j)|^2\bigg)dxds \\
  &\phantom{xx}+ C\int_0^t\int_\Omega
  \sum_{k=0}^n(u_k-\widehat u_k)^2|\na\log\widehat u_k|^2 dxds.
\end{align*}

The term $I_2$ is decomposed into several parts. We write $h_2(u)=\sum_{i=0}^n u_i(\log u_i-1)$. Observing that $\bar B^{21}=0$ (with a notation similar as in \eqref{3.barQ}), 
we find that $I_2=I_{21}+I_{22}+I_{23}$, where
\begin{align*}
  I_{21} &= -\int_0^t\int_\Omega\big[
  (\na(h_2'(u)-h_2'(\widehat u)))^T
  :\bar B^{11}(\bar q) - (\na(h_2''(\widehat{u})
  (u-\widehat u)))^T:\bar B^{11}(\widehat{\bar q})\big]
  \na h_2'(\widehat{\bar q})dxds, \\
  I_{22} &= -\int_0^t\int_\Omega\big[
  (\na(h_2'(u)-h_2'(\widehat u)))^T
  :\bar B^{12}(\bar q) - (\na(h_2''(\widehat{ u})
  (u-\widehat u)))^T:\bar B^{12}(\widehat{\bar q})\big]
  K\na\widehat c dxds, \\
  I_{23} &= -\int_0^t\int_\Omega\bigg(K(\na(c-\widehat c))^T
  :\bar B^{22}(\bar q) - K(\na(c-\widehat c))^T:
  \bar B^{22}(\widehat{\bar q})\bigg)K\na\widehat c dxds.
\end{align*} 
Since $\bar B^{22}_{ij}(\bar q)=(D_i/K)\delta_{ij}$, the last term vanishes, $I_{23}=0$. We infer from \eqref{4.uz} and definition \eqref{3.barQ} of $\bar Q^{11}$ that
\begin{align*}
  I_{21} &= -\sum_{i,j=0}^n\int_0^t\int_\Omega\bigg(
  u_j\bar Q_{ij}^{11}(\bar q)\na\log\frac{u_i}{\widehat u_i}
  - \widehat u_j\bar Q_{ij}^{11}(\widehat{\bar q})
  \na\bigg(\frac{1}{\widehat u_i}(u_i-\widehat u_i)\bigg)
  \bigg)\cdot\na\log \widehat u_j dxds \\
  &= -\sum_{i,j=0}^n\int_0^t\int_\Omega\bigg(
  u_j\bar Q_{ij}^{11}(\bar q)\na\log\frac{u_i}{\widehat u_i}
  - \widehat u_j\bar Q_{ij}^{11}(\widehat{\bar q})
  \frac{u_i}{\widehat u_i}
  \na\log\frac{u_i}{\widehat u_i}\bigg)
  \cdot\na\log \widehat u_j dxds \\
  &= -\sum_{i,j=0}^n\int_0^t\int_\Omega\bigg(
  (\alpha_iu_i^{m_i-1}\delta_{ij} - u_i\alpha_j u_j^{m_j-1})
  \frac{u_j}{\widehat u_j} - (\alpha_i\widehat u_i^{m_i-1}\delta_{ij}
  - \widehat u_i\alpha_j \widehat u_j^{m_j-1})
  \frac{u_i}{\widehat u_i}\bigg) \\
  &\phantom{xx}\times
  \na\log\frac{u_i}{\widehat u_i}\cdot\na \widehat u_j dxds \\
  &= -\sum_{i=0}^n\int_0^t\int_\Omega\alpha_i u_i
  (u_i^{m_i-1}-\widehat u_i^{m_i-1})
  \na\log\frac{u_i}{\widehat u_i}\cdot\na\log\widehat u_i dxds \\
  &\phantom{xx}+ \sum_{i,j=0}^n\int_0^t\int_\Omega
  \alpha_j u_i(u_j^{m_j}-\widehat u_j^{m_j})
  \na\log\frac{u_i}{\widehat u_i}\cdot\na\log \widehat u_j dxds.
\end{align*}
We split the term $\na\log(u_i/\widehat u_i)$ into two parts, combining \eqref{4.zPi} and \eqref{4.uz} for $i=0,\ldots,n$:
\begin{align}\label{4.nalog}
  \na\log\frac{u_i}{\widehat u_i}
  &= \Pi_1\bigg(\na\log\frac{u}{\widehat u}\bigg)_i 
  + (I-\Pi_1)\bigg(\na\log\frac{u}{\widehat u}\bigg)_i \\
  &= \Pi_1\bigg(\na\log\frac{u}{\widehat u}\bigg)_i 
  + \sum_{k=0}^n(\widehat u_k-u_k)\na\log\widehat u_k, \nonumber 
\end{align}
where $\log(u/\widehat u)$ denotes the vector $(\log(u_i/\widehat u_i))_{i=0}^n$ and recalling that $\Pi_1(\bar y)_i=y_i - \sum_{k=0}^n u_ky_k$ for $i=0,\ldots,n$. Since $\Pi_1(\na\log(u/\widehat u),c)\in L$, we have $\sum_{i=0}^n u_i\Pi_1(\na\log(u/\widehat u)_i)=0$ and therefore,
\begin{align*}
  I_{21} &= -\sum_{i=0}^n\int_0^t\int_\Omega\alpha_i u_i
  (u_i^{m_i-1}-\widehat u_i^{m_i-1})\na\log\widehat u_i\cdot
  \Pi_1\bigg(\na\log\frac{u}{\widehat u}\bigg)_i dxds \\
  &\phantom{xx}- \sum_{i=0}^n\int_0^t\int_\Omega\alpha_i u_i
  (u_i^{m_i-1}-\widehat u_i^{m_i-1})\na\log\widehat u_i\cdot
  \bigg(\sum_{k=0}^n(\widehat u_k-u_k)\na\log\widehat u_k\bigg)dxds \\
  &\phantom{xx}+ \int_0^t\int_\Omega
  \bigg(\sum_{j=0}^n\alpha_j(u_j^{m_j}-\widehat u_j^{m_j})
  \na\log\widehat u_j\bigg)
  \cdot\bigg(\sum_{i=0}^n(\widehat u_i-u_i)\na\log\widehat u_i\bigg)
  dxds.
\end{align*}
We recall the definition $z_i=\na\log(u_i/\widehat u_i)$ for $i=0,\ldots,n$ and apply Young's inequality to infer that
\begin{align*}
  I_{21} &\le \int_0^t\int_\Omega
  \sum_{i=0}^n\frac{\alpha_i}{8}
  u_i^{m_i}|\Pi(\bar z)_i|^2 dxds
  + C\int_0^t\int_\Omega\sum_{i=0}^n
  |u_i^{m_i-1}-\widehat u_i^{m_i-1}|^2|\na\log\widehat u_i|^2 dxds \\
  &\phantom{xx}+ C\int_0^t\int_\Omega\sum_{i=0}^n
  \big(|u_i-\widehat u_i|^2|\na\log\widehat u_i|^2 
  + |u_i^{m_i}-\widehat u_i^{m_i}|^2|\na\log\widehat u_i|^2 \big)dxds.
\end{align*}
Since $\widehat u_i$ is assumed to be strictly positive, $\widehat u_i\ge\delta>0$, we can estimate $|u_i^{m_i-1}-\widehat u_i^{m_i-1}|\le C(\delta)|u_i-\widehat u_i|$ and $|u_i^{m_i}-\widehat u_i^{m_i}|\le C|u_i-\widehat u_i|$, which leads to
\begin{align*}
  I_{21} \le \int_0^t\int_\Omega\sum_{i=0}^n\frac{\alpha_i}{8}
  u_i^{m_i}|\Pi(z)_i|^2 dxds
  + C\int_0^t\int_\Omega\sum_{i=0}^n
  |u_i-\widehat u_i|^2|\na\log\widehat u_i|^2 dxds.
\end{align*}

Next, we estimate $I_{22}$ by substituting definition \eqref{3.barQ} of $\bar Q^{12}$ and using \eqref{4.uz}:
\begin{align*}
  I_{22} &= -\sum_{i=0}^n\sum_{k=1}^g\int_0^t\int_\Omega
  u_i\bigg(\chi_{ik}-\sum_{j=0}^n u_j\chi_{jk}\bigg)
  \na\widehat c_k\cdot\nabla\log\frac{u_i}{\widehat u_i}dxds \\
  &\phantom{xx}+ \sum_{i=0}^n\sum_{k=1}^g\int_0^t\int_\Omega
  u_i\bigg(\chi_{ik}-\sum_{j=0}^n\widehat u_j\chi_{jk}\bigg)
  \na\widehat c_k\cdot\nabla\log\frac{u_i}{\widehat u_i}dxds \\
  &= \int_0^t \int_\Omega \bigg( \sum_{i=0}^n(\widehat u_i -u_i)
  \na\log\widehat u_i\bigg) \cdot \sum_{j=0}^n\sum_{k=1}^g
  \chi_{jk}(u_j - \widehat u_j)\nabla \widehat c_k dxds.
\end{align*}
Since $\widehat u \in L^\infty(\Omega_T)$ and $\widehat c^0\in W^{1,q}(\Omega)$ with $q>d$, parabolic regularity yields $\widehat c \in L^\infty(0,T;$ $W^{2,q}(\Omega))\hookrightarrow L^\infty(0,T;W^{1,\infty}(\Omega))$. Then we infer from Young's inequality that
\begin{align*}
    I_{22} \leq \int_0^t \int_\Omega\sum_{i=0}^n|u_i - \widehat u_i|^2 |\nabla \log \widehat u_i|^2 dx ds + C \int_0^t\int_\Omega \sum_{i=0}^n |u_i - \widehat u_i|^2 dx ds.
\end{align*}

For the term $I_3$, we substitute the explicit form of $r(q)$ and apply Young's inequality again. This leads to
\begin{align*}
  I_3 &= K\int_0^t\int_\Omega\sum_{j=1}^g\bigg(
  -\lambda_j(c_j-\widehat c_j)^2 + \sum_{i=0}^n b_{ji}(u_i-\widehat u_i)
  (c_j-\widehat c_j)\bigg)dxds \\
  &\le CK\int_0^t\int_\Omega\sum_{i=0}^n |u_i-\widehat u_i|^2dxds
  - CK\int_0^t\int_\Omega\sum_{j=1}^g |c_j-\widehat c_j|^2dxds.
\end{align*}
We combine the estimates for $I_1$, $I_2$, and $I_3$ to conclude from \eqref{4.I123} that
\begin{align*}
  \int_\Omega& h(\bar{q}|\widehat{\bar q})(t)dx - \int_\Omega h(\bar{q}|\widehat{\bar q})(0)dx
  + CK\int_0^t\int_\Omega\sum_{j=1}^g|c_j-\widehat c_j|^2 dxds \\
  &\phantom{xx} + \int_0^t\int_\Omega
  \sum_{i=0}^n\frac{\alpha_i}{8} u_i^{m_i}|\Pi(z)_i|^2 dxds
  + \gamma(K)\int_0^t\int_\Omega 
  \sum_{j=1}^g|\na(c_j-\widehat c_j)|^2 dxds \\
  &\le C(K+1)\int_0^t\int_\Omega\sum_{i=0}^n|u_i-\widehat u_i|^2 dxds 
  + C\int_0^t\int_\Omega\sum_{j=1}^g|\na(c_j-\widehat c_j)|^2 dxds.
\end{align*}
We choose $K$ sufficiently large such that the last term on the right-hand side is absorbed by the corresponding term on the left-hand side. This is possible since $\gamma(K) \to \infty$ as $K\to\infty$. It follows that
\begin{align*}
  \int_\Omega h(\bar{q}|\widehat{\bar q})(t)dx 
  - \int_\Omega h(\bar{q}|\widehat{\bar q})(0)dx
  \le C(K+1)\int_0^t\int_\Omega\sum_{i=0}^n|u_i-\widehat u_i|^2 dxds .
\end{align*}
The inequality $u_i\log(u_i/\widehat u_i) - u_i + \widehat u_i \ge \frac12|u_i-\widehat u_i|^2$ for $u_i$, $\widehat u_i\in[0,1]$, proven in \cite[Lemma 16]{HJT22}, implies that
\begin{align}\label{5.L2}
  \frac12\sum_{i=0}^n\int_\Omega|u_i-\widehat u_i|^2 dx
  + \frac{K}{2}\sum_{j=1}^g\int_\Omega|c_j-\widehat c_j|^2 dx
  \le \int_\Omega h(\bar q|\widehat{\bar q})dx.
\end{align}
This shows that
\begin{align*}
  \int_\Omega h(\bar{q}|\widehat{\bar q})(t)dx 
  \le \int_\Omega h(\bar{q}|\widehat{\bar q})(0)dx
  + C(K+1)\int_0^t\int_\Omega h(\bar{q}|\widehat{\bar q})(s)dxds .
\end{align*}
We deduce from Gronwall's inequality that
\begin{align*}
  \int_\Omega h(\bar{q}|\widehat{\bar q})(t)dx 
  \le e^{C(K+1)t}\int_\Omega h(\bar{q}|\widehat{\bar q})(0)dx,
\end{align*}
and since the initial data coincide, $h(\bar{q}|\widehat{\bar q})(0)=0$, we obtain $h(\bar{q}|\widehat{\bar q})(t)=0$ and consequently $u(t)=\widehat u(t)$, $c(t)=\widehat c(t)$ in $\Omega$ for $t>0$, concluding the proof.


\section{Proof of the exponential decay of solutions}\label{sec.time}
We wish to use the test function $h_2'(\bar q)=(\log u_0,\ldots,\log u_n,Kc_1,\ldots,Kc_g)$ in the weak formulation of \eqref{1.u}--\eqref{1.c}. Since $u_i$ may vanish, we need to regularize and choose the test function $(\log(u_0+\eps),\ldots,\log(u_n+\eps),Kc_1,\ldots,Kc_g)$ with $\eps>0$. A computation and the limit $\eps\to 0$, similarly as in \cite[Sec.~3.1]{ChJu19}, then leads for $0<s<t<T$ to
\begin{align*}
  \int_\Omega& h(\bar{q}(t)|\bar{q}^\infty)dx 
  - \int_\Omega h(\bar{q}(s)|\bar{q}^\infty)dx \\
  &= -\sum_{i,j=0}^n\int_s^t\int_\Omega
  \bar Q_{ij}^{11}\nabla\log u_i\cdot\nabla u_jdxd\sigma
  -\sum_{i=0}^n\sum_{j=1}^g\int_s^t\int_\Omega
  \bar Q_{ij}^{12}\nabla\log u_i\cdot\nabla c_jdxd\sigma \\
  &\phantom{xx}-K\sum_{i,j=1}^g\int_s^t\int_\Omega
  \bar Q_{ij}^{22}\nabla c_i\cdot\nabla c_j dxd\sigma 
  + K\int_s^t\int_\Omega(r(\bar q)-r(\bar{q}^\infty))
  \cdot(c-c^\infty)dxd\sigma \\
  &= -\int_s^t\int_\Omega\bigg(\sum_{i=0}^n \alpha_i u_i^{m_i-2}
  |\na u_i|^2 + \sum_{i=0}^n\sum_{j=1}^g\chi_{ij}\na u_i\cdot\na c_j
  + K\sum_{j=1}^g D_j|\na c_j|^2\bigg)dxd\sigma \\
  &\phantom{xx}+ K\int_s^t\int_\Omega\sum_{j=1}^g\bigg(
  \sum_{k=0}^n b_{jk}(u_k-u_k^\infty) 
  - \lambda_j(c_j-c_j^\infty)\bigg)(c_j-c_j^\infty)dxd\sigma.
\end{align*}
We use $\na u_0=-\sum_{i=1}^n\na u_i$ to find that
\begin{align*}
  \sum_{i=0}^n\chi_{ij}\na u_i\cdot\na c_j
  = \sum_{i=1}^n(\chi_{ij}-\chi_{0j})\na u_i\cdot\na c_j.
\end{align*}
Then we apply Young's inequality with $\eps>0$ to this expression and the property $u_i^{m_i-2}\ge 1$ (since $1\le m_i\le 2$):
\begin{align}\label{4.aux}
  \int_\Omega& h(\bar{q}(t)|\bar{q}^\infty)dx 
- \int_\Omega h(\bar{q}(s)|\bar{q}^\infty)dx \\
  &\le \int_s^t\int_\Omega\bigg(
  \sum_{i=1}^n\bigg(\frac{\eps}{2}-\alpha_i\bigg)|\na u_i|^2
  + \sum_{j=1}^g\bigg(\frac{\|\chi'\|_2^2}{2\eps} - KD_j\bigg)
  |\na c_j|^2\bigg)dxd\sigma \nonumber \\
  &\phantom{xx}+ K\int_s^t\int_\Omega\sum_{j=1}^g\bigg(
  \sum_{k=0}^n b_{jk}(u_k-u_k^\infty) 
  - \lambda_j(c_j-c_j^\infty)\bigg)(c_j-c_j^\infty)dxd\sigma, \nonumber 
\end{align}
where $\|\chi'\|_2$ is the spectral norm of the matrix $\chi'=(\chi_{ij}-\chi_{0j})_{ij}$. To estimate the last term, we write
\begin{align*}
  \sum_{i=0}^n b_{ji}(u_i-u_i^\infty)
  = \sum_{i=1}^n(b_{ji}-b_{j0})(u_i-u_i^\infty).
\end{align*}
Thus, by Young's inequality with $\delta>0$:
\begin{align*}
  K&\int_s^t\int_\Omega\sum_{j=1}^g\bigg(
  \sum_{k=0}^n b_{jk}(u_k-u_k^\infty) 
  - \lambda_j(c_j-c_j^\infty)\bigg)(c_j-c_j^\infty)dxd\sigma \\
  &\le \int_s^t\int_\Omega\bigg(
  \delta\sum_{i=1}^n(u_i-u_i^\infty)^2 
  + K\sum_{j=1}^g\bigg(\frac{\|b'\|_2^2}{4\delta}K - \lambda_j\bigg)
  (c_j-c_j^\infty)^2\bigg)dxd\sigma,
\end{align*}
where $b'=(b_{ji}-b_{j0})_{ij}$. Hence, we obtain from \eqref{4.aux} and the Poincar\'e inequality with constant $C_P>0$ (if $\eps/2<\alpha_i$ for $i=1,\ldots,n$) that
\begin{align*}
  \int_\Omega& h(\bar{q}(t)|\bar{q}^\infty)dx 
  - \int_\Omega h(\bar{q}(s)|\bar{q}^\infty)dx \\
  &\le \int_s^t\int_\Omega\bigg(
  \sum_{i=1}^n\bigg(\frac{\eps}{2}-\alpha_i\bigg)
  C_P^{-2}(u_i-u_i^\infty)^2 
  + \bigg(\frac{\|\chi'\|_2^2}{2\eps} - KD_*\bigg)
  \sum_{j=1}^g|\na c_j|^2\bigg)dxd\sigma \\
  &\phantom{xx}+ \int_s^t\int_\Omega\bigg(
  \delta\sum_{i=1}^n(u_i-u_i^\infty)^2 
  + K\bigg(\frac{\|b'\|_2^2}{4\delta}K-\lambda_*\bigg)
  \sum_{j=1}^g(c_j-c_j^\infty)^2\bigg)dxd\sigma.
\end{align*}

Next, we choose $\delta=\eps/(2C_P^2)$, $K=\|\chi'\|_2^2/(2\eps D_*)$. Then
\begin{align*}
  \int_\Omega h(\bar{q}(t)|\bar{q}^\infty)dx 
  &\le \int_\Omega h(\bar q(s)|\bar q^\infty)dx
  + \frac{\eps-\alpha_*}{C_P^{2}}\int_s^t\int_\Omega
  \sum_{i=1}^n(u_i-u_i^\infty)^2 dxd\sigma \\
  &\phantom{xx}+ K\bigg(\frac{C_P^2\|b'\|_2^2\|\chi'\|_2^2}{4\eps^2
  D_*} - \lambda_*\bigg)\int_s^t\int_\Omega\sum_{j=1}^g
  (c_j-c_j^\infty)^2 dxd\sigma.
\end{align*}
Dividing this inequality by $t-s$ and passing to the limit $s\to t$, we arrive at
\begin{align*}
  \frac{d}{dt}&\int_\Omega h(\bar{q}|\bar{q}^\infty)dx
  \le -\frac{\alpha_*-\eps}{C_P^2}\int_\Omega
  \sum_{i=1}^n(u_i-u_i^\infty)^2 dx \\
  &- K\bigg(\lambda_* - \frac{C_P^2\|b'\|_2^2\|\chi'\|_2^2}{4\eps^2
  D_*}\bigg)\int_\Omega\sum_{j=1}^g(c_j-c_j^\infty)^2 dx
  \le -\mu\int_\Omega h(\bar q|\bar q^\infty)dx,
\end{align*}
where we used in the next-to-last step the inequality
\begin{align*}
  \frac{1}{2} \sum_{i = 0}^{n} 
  \| u_i - u^\infty_i \|_{L^{2}(\Omega)}^2
  \leq \int_{\Omega} \sum_{i = 0}^{n} 
  \bigg( u_i \log \frac{u_i}{{u}^\infty_i} - u_i + {u}^\infty_i\bigg)dx
  \leq C_u \sum_{i = 1}^{n} \| u_i - {u}^\infty_i \|_{L^{2}(\Omega)}^2,
\end{align*}
and we have set
\begin{align*}
  \mu = \min\bigg\{\frac{\alpha_*-\eps}{C_P^2 C_u},
  2\bigg(\lambda_*-\frac{C_P^2\|b'\|_2^2\|\chi'\|_2^2}{4\eps^2 D_*}
  \bigg)\bigg\}>0.
\end{align*}
Gronwall's lemma implies that
\begin{align*}
  \int_\Omega h(\bar q(t)|\bar q^\infty)dx
  \le e^{-\mu t}\int_\Omega h(\bar q(0)|\bar q^\infty)dx
  \quad\mbox{for }t>0,
\end{align*}
and inequality \eqref{5.L2} completes the proof.


\section{Proof of the vanishing diffusion limit}\label{sec.lim}

Let $q^\eps=(u^\eps,c^\eps)$ be a weak solution to \eqref{1.u}--\eqref{1.bc} with $D_j=\eps$ for $j=1,\ldots,g$. We set $\bar q^\eps=(u_0^\eps,u^\eps,c^\eps)$. We conclude from $\pa_t c_j^\eps - D_j\Delta c_j^\eps + \lambda_j c_j^\eps\in L^2(\Omega_T)$ and parabolic regularity that $c^\eps$ is a strong solution to \eqref{1.c}. The entropy density $h_3$, defined in \eqref{1.h3}, can be written as
\begin{align*}
  h_3(\bar q^\eps) = h_2(\bar q^\eps) 
  - \frac{K}{2}\sum_{j=1}^g (c_j^\eps)^2
  + \frac{K}{2}\sum_{j=1}^g|\na c_j^\eps|^2 + n + 1.
\end{align*}
We compute
\begin{align*}
  \frac12\frac{d}{dt}\int_\Omega|\na c_j^\eps|^2 dx
  &= -\int_\Omega\pa_t c_j^\eps\Delta c_j^\eps dx \\
  &= -\int_\Omega(\eps(\Delta c_j^\eps)^2 + \lambda_j|\na c_j^\eps|^2)dx
  + \int_\Omega\sum_{k=0}^n b_{jk}\na u_k^\eps\cdot\na c_j^\eps dx.
\end{align*}
Hence, taking into account the entropy inequality \eqref{4.ei} for $h_2$, applying Young's inequality to the mixed term $\na u_k^\eps\cdot\na c_j^\eps$, and using $(u_i^\eps)^{m_i-2}\ge 1$ as well as $\sum_{k=0}^n b_{jk}\nabla u_k^\eps
=\sum_{k=1}^n (b_{jk}-b_{j0})\na u_k^\eps$, we obtain the following entropy inequality for $h_3$:
\begin{align*}
  \int_\Omega &h_3(\bar q^\eps(t))dx
  + \int_0^t\int_\Omega\bigg(\sum_{i=0}^n
  \frac{\alpha_i}{2}|\na u_i^\eps|^2
  + K\eps\sum_{j=1}^g(\Delta c_j^\eps)^2\bigg)dxds \\
  &\le \int_\Omega h_3(\bar q(0))dx 
  + C(b,\chi)\int_0^t\int_\Omega\sum_{j=1}^g|\na c_j^\eps|^2 dxds \\
  &\le \int_\Omega h_3(\bar q(0))dx 
  + C(b,\chi)\int_0^t\int_\Omega h_3(\bar q^\eps)dxds.
\end{align*}
An application of Gronwall's lemma shows that there exists $C>0$ independent of $\eps$ such that
\begin{align*}
  \|u_i^\eps\|_{L^\infty(\Omega_T)}
  + \|u_i^\eps\|_{L^2(0,T;H^1(\Omega))}
  + \|\na c_j^\eps\|_{L^\infty(0,T;L^2(\Omega))} 
  + \sqrt\eps\|\Delta c_j^\eps\|_{L^2(\Omega_T)} \le C,
\end{align*}
where $i=0,\ldots,n$, $j=1,\ldots,g$. Observe that the estimate for $u_0^\varepsilon$ follows from the volume-filling constraint. The $L^\infty(\Omega_T)$ bound for $(u_i^\eps)$ implies an $L^\infty(0,T;L^2(\Omega))$ bound for $c_j^\eps$. This yields a uniform $L^\infty(0,T;H^1(\Omega))$ bound for $c_j^\eps$. From these estimates, we infer uniform bounds for $\pa_t u_i^\eps$ and $\pa_t c_j^\eps$ in $L^2(0,T;H^1(\Omega)')$. Thus, we can extract a subsequence (not relabeled) such that
\begin{align*}
  u_i^\eps\to u_i, \quad c_j^\eps\to c_j
  \quad\mbox{strongly in }L^2(\Omega_T),\ i=0,\ldots,n,\ j=1,\ldots,g.
\end{align*}
Moreover, we have the weak convergences
\begin{align*}
  \na u_i^\eps\rightharpoonup\na u_i, \quad
  \na c_j^\eps\rightharpoonup\na c_j
  &\quad\mbox{weakly in }L^2(\Omega_T), \\
  \pa_t u_i^\eps\rightharpoonup\pa_t u_i, \quad
  \pa_t c_j^\eps\rightharpoonup\pa_t c_j
  &\quad\mbox{weakly in }L^2(0,T;H^1(\Omega)'), \\
  \eps\Delta c_j^\eps\to 0 &\quad\mbox{strongly in }L^2(\Omega_T).
\end{align*}
Thanks to the $L^\infty(\Omega_T)$ bound for $u_i^\eps$, the family $(u_i^\eps)^{m_i-1}$ converges strongly in any $L^r(\Omega_T)$ with $r<\infty$. Thus, we can pass to the limit $\eps\to 0$ in \eqref{1.u}--\eqref{1.bc} with $D_j=\eps$ to find that $(u,c)$ solves \eqref{1.lim}. 


\begin{appendix}
\section{Formal derivation of the model}\label{sec.deriv}

We derive model \eqref{1.u}--\eqref{1.c} from mass and force balance equations using the multiphase approach of \cite{LKBJS06}. Let $u_i$ and $v_i$ be the volume fraction and velocity of the $i$th phase for $i=0,\ldots,n$. We assume that the phases are incompressible and that the densities are constant and have the same value, since the cell components mainly consist of water. Neglecting inertia effects, the mass and momentum balance equations then read as
\begin{align}\label{2.mmb}
  \pa_t u_i + \diver(u_iv_i) = 0, \quad
  0 = -\diver(u_ip_i(u)I) + u_iF_i, \quad i=0,\ldots,n,
\end{align}
where $p_i(u)$ is the phase-specific pressure and $F_i$ is the interaction force of the $i$th phase. We suppose that the barycentric velocity vanishes; this implies that $\sum_{i=0}^n u_iv_i=0$. Thus, if the initial fractions satisfy $\sum_{i=0}^n u_i(0)=1$, a summation of the first equation in \eqref{2.mmb} leads to $\pa_t\sum_{i=0}^n u_i=0$ and hence to $\sum_{i=0}^n u_i(t)=1$ for all times $t\ge 0$. Conversely, if the volume-filling constraint $\sum_{i=0}^n u_i=1$ holds then the barycentric velocity $\sum_{i=0}^n u_iv_i$ is divergence-free and the fluid is incompressible. In this sense, the volume-filling constraint is equivalent to the incompressibility of the mixture.

We specify the pressures and forces. The pressure $p_i$ is the sum of the overall pressure $p$ common to all mixture components and the intraphase pressure $\pi_i(u)$:
\begin{align*}
  p_i = p + \pi_i(u), \quad \pi_i(u) = \frac{\alpha_i}{m_i} u_i^{m_i-1}
  \quad\mbox{for }i=0,\ldots,n,
\end{align*}
where $m_i\ge 1$ is a porous-medium-type exponent and $\alpha_i>0$ for $i=1,\ldots,n$, $\alpha_0\ge 0$. The interaction force is the sum of the interphase forces $f_{ij}$, the viscous drag forces $g_{ij}$, and the chemotactic forces $k_i$. Following \cite{LKBJS06} and neglecting the pressure due to the traction between the phases, the interphase forces read as $f_{ij} = p(u_j\na u_i - u_i\na u_j)$. Then, taking into account that $\sum_{i=0}^n u_i=1$ and $\sum_{i=0}^n\na u_i=0$,
\begin{align}\label{2.sumf}
  \sum_{j=0}^n f_{ij} = p\na u_i.
\end{align}
The drag forces take the form $g_{ij} = -b_{ij}u_iu_j(v_i-v_j)$, where $b_{ij}=b_{ji}$ are some parameters. We simplify the equations by assuming that $b_{ij}=1$ and using $\sum_{j=0}^n u_jv_j=0$, which leads to
\begin{align}\label{2.sumg}
  \sum_{j=0}^n g_{ij} = -u_iv_i\sum_{j=0}^n u_j 
  + u_i\sum_{j=0}^n u_jv_j = -u_iv_i.
\end{align}
To determine the chemotactic force, let $c_1,\ldots,c_g$ be chemical concentrations solving 
\begin{align}\label{2.c}
  \pa_t c_j = D_j\Delta c_j - \lambda_jc_j 
  + \sum_{k=0}^n b_{jk}u_k, \quad j=1,\ldots,g,
\end{align}
where $D_j$, $\lambda_j>0$ and $b_{jk}\ge 0$. The cell phase $u_i$ responds to a chemical cue with a force $f_i$, which acts on the constituents of the mixture and is proportional to the $i$th volume fraction and the concentration gradients $\na c_j$ for $j=1,\ldots,g$. The force is distributed among the constituents according to their volume fraction. Therefore, the chemotactic force is given by 
\begin{align}\label{2.k} 
  k_i = f_i - u_i\sum_{j=0}^n f_j, 
  \quad\mbox{where }f_i=-u_i\sum_{j=1}^g\chi_{ij}\na c_j.
\end{align}

Collecting identities \eqref{2.sumf}, \eqref{2.sumg}, and \eqref{2.k}, the interaction forces become
\begin{align*}
  u_iF_i = \sum_{j=0}^n(f_{ij}+g_{ij}) + k_i
  = p\na u_i - u_iv_i + f_i - u_i\sum_{j=0}^n f_j.
\end{align*}
We insert this expression into \eqref{2.mmb}. Then
\begin{align}\label{2.aux}
0&=-\nabla(u_i(p+\pi_i(u)))+p\nabla u_i-u_iv_i
+f_i-u_i\sum_{j=0}^n f_j\\
&=-u_i\nabla p-\nabla(u_i\pi_i(u))-u_iv_i
+f_i-u_i\sum_{j=0}^n f_j.\nonumber
\end{align}
The term $p\na u_i$ cancels, and after summing these equations over $i=0,\ldots,n$, we obtain $\na p = -\sum_{i=0}^n\na(u_i\pi_i(u))$. Thus, we can replace $\na p$ in \eqref{2.aux} to find that
\begin{align*}
  u_iv_i = -\na(u_i\pi_i(u)) + u_i\sum_{j=0}^n\na(u_j\pi_j(u))
  + f_i - u_i\sum_{j=0}^n f_j.
\end{align*}
We insert the flux $u_iv_i$ into the mass balance equation:
\begin{align*}
  & \pa_t u_i = \diver\bigg(\big(\na(u_i\pi_i(u))-f_i\big)
  - u_i\sum_{j=0}^n\big(\na(u_j\pi_j(u))-f_j\big)\bigg), \\
  & \mbox{where}\ f_i=-u_i\sum_{j=1}^g\chi_{ij}\na c_j, 
  \quad \pi_i(u)=\frac{\alpha_i}{m_i} u_i^{m_i-1}, \quad i=0,\ldots,n.
\end{align*}
Introducing $J_i=\na(u_i\pi_i(u))-f_i$, we recover equation \eqref{1.u}, while equation \eqref{1.c} for the chemical concentrations equals \eqref{2.c}.  

\end{appendix}


\end{document}